\documentclass[a4paper,10pt]{article}
\usepackage{stmaryrd}
\usepackage{amsfonts}
\usepackage{bbm}
\usepackage{amscd}
\usepackage{mathrsfs}
\usepackage{latexsym,amssymb,amsmath,amscd,amscd,amsthm,amsxtra}
\usepackage[dvips]{graphicx}
\usepackage[utf8]{inputenc}
\usepackage[T1]{fontenc}
\usepackage{lmodern}
\usepackage{amssymb}
\usepackage[all]{xy}
\usepackage{nicefrac,mathtools,enumitem}
\usepackage{microtype}
\usepackage{xcolor}
\newtheorem{thm}{Theorem}[section]
\newtheorem{lem}[thm]{Lemma}
\newtheorem{cor}[thm]{Corollary}
\newtheorem{pro}[thm]{Proposition}
\newtheorem{ex}[thm]{Example}
\newtheorem{rmk}[thm]{Remark}
\newtheorem{defi}[thm]{Definition}

\newcommand {\emptycomment}[1]{}

\newcommand{\be }{\begin{equation}}
\newcommand{\ee }{\end{equation}}
\newcommand{\K}{\mathbb{K}}

\newcommand{\pf}{\noindent{\bf Proof.}\ }
\newcommand {\yh}[1]{{\marginpar{*}\scriptsize\textcolor{purple}{yh: #1}}}

\newcommand{\dr}{\dM^{\rm reg}}

\newcommand{\Real}{\mathbb R}

\newcommand{\Nat}{\mathbb N}
\newcommand{\Integ}{\mathbb Z}
\newcommand{\Field}{\mathbb F}
\newcommand{\huaR}{\mathcal{R}}

\newcommand{\huaO}{\mathcal{O}}

\newcommand{\huaN}{\mathcal{N}}

\newcommand{\g}{\mathfrak g}

\newcommand{\frkg}{\mathfrak g}

\newcommand{\frks}{\mathfrak s}

\newcommand{\frkB}{\mathfrak B}

\def\qed{\hfill ~\vrule height6pt width6pt depth0pt}

\newcommand{\half}{\frac{1}{2}}

\newcommand{\Id}{\rm{Id}}

\newcommand{\e}{\mathbbm{e}}

\newcommand{\br}[1]{   [ \cdot,    \cdot  ]   }

\newcommand{\dM}{\mathrm{d}}

\newcommand{\Hom}{\mathrm{Hom}}

\newcommand{\gl}{\mathfrak {gl}}

\newcommand{\Sym}{\mathrm {Sym}}

\newcommand{\ad}{\mathrm{ad}}

\newcommand{\Img}{\mathrm{Im}}

\newcommand{\sgn}{\mathrm{sgn}}

\begin{document}
\title{
{Nijenhuis operators on pre-Lie algebras\thanks{Research supported
by NSFC (11471139,  11425104), NSF of Jilin Province
(20170101050JC) and Nanhu Scholars Program for Young Scholars of
XYNU. }}
 }\vspace{2mm}
\author{Qi Wang, Chengming Bai, Jiefeng Liu and Yunhe Sheng
}

\date{}

\maketitle

\begin{abstract}
First we use a new approach to give a graded Lie algebra whose
Maurer-Cartan elements characterize pre-Lie algebra structures.
Then using this graded Lie bracket we define the notion of a
Nijenhuis operator on a pre-Lie algebra which generates a trivial
deformation of this pre-Lie algebra. There are close relationships
between $\huaO$-operators, Rota-Baxter operators and Nijenhuis
operators on a pre-Lie algebra. In particular, a Nijenhuis
operator  ``connects'' two $\huaO$-operators on a pre-Lie algebra
whose any linear combination is still an $\mathcal O$-operator in
certain sense and hence compatible L-dendriform algebras appear
naturally as the induced algebraic structures. For the case of the
dual representation of the regular representation of a pre-Lie
algebra, there is a geometric interpretation by introducing the
notion of a pseudo-Hessian-Nijenhuis structure which  gives
rise to a sequence of pseudo-Hessian and pseudo-Hessian-Nijenhuis
structures. Another application of Nijenhuis operators on pre-Lie
algebras in geometry is illustrated by introducing the notion of a
para-complex structure on a pre-Lie algebra and then studying
paracomplex quadratic pre-Lie algebras and paracomplex
pseudo-Hessian pre-Lie algebras in detail. Finally, we give some
examples of Nijenhuis operators on pre-Lie algebras.
\end{abstract}

\tableofcontents

\section{Introduction}

For an algebra $\g$, constructing a graded Lie algebra structure on the standard complex of
multilinear maps (with possible additional conditions) of $\g$ into itself is very important. Deformations and cohomologies of the algebra can be well studied using this graded Lie algebra. For an associative algebra $\g$, the corresponding graded Lie algebra is given as follows.
On the graded vector space $\oplus_k\Hom(\otimes^{k+1}\g,\g)$, there is a graded commutator bracket
$$
[P,Q]^G=P\circ Q-(-1)^{pq}Q\circ P,
$$
  where $P\in\Hom(\otimes^{p+1}\g,\g),
~Q\in \Hom(\otimes^{q+1}\g,\g)$,  such that $(\oplus_k\Hom(\otimes^{k+1}\g,\g),[\cdot,\cdot]^G)$  is a graded Lie algebra.
 This bracket is usually called the  Gerstenhaber bracket. Note that the operation $\circ$ defines a graded pre-Lie
algebra structure.
 A bilinear map $\pi\in\Hom(\otimes^2\g,\g)$ is an associative algebra if and only if $[\pi,\pi]^G=0$
 and the coboundary operator for the cohomology theory of
this associative algebra is
 defined by $\delta(f)=(-1)^{|f|}[\pi,f]^G$.  See \cite{Gerstenhaber1,Gerstenhaber2,Nij,NijRic} for details of graded Lie algebras for associative algebras, commutative algebras and Lie algebras.

  In \cite{ChaLiv}, the authors constructed such a graded
Lie algebra for pre-Lie algebras  using the operad theory. Pre-Lie
algebras are a class of nonassociative algebras coming from the
study of convex homogeneous cones, affine manifolds and affine
structures on Lie groups, and the aforementioned cohomologies of
associative algebras.  They also appeared in many fields in
mathematics and mathematical physics, such as complex and
symplectic structures on Lie groups and Lie algebras, integrable
systems, Poisson brackets and infinite dimensional Lie algebras,
vertex algebras, quantum field theory and operads. See
\cite{Andrada,Bakalov,ChaLiv,Lichnerowicz}, and the survey
\cite{Pre-lie algebra in geometry} and the references therein for
more details.

In this paper, we use a different approach to construct the graded Lie algebra for pre-Lie algebras. We introduce a partial skew-symmetrization operator $\alpha:\Hom(\otimes^{k+1}\g,\g)\longrightarrow \Hom(\wedge^{k}\g\otimes\g,\g)$. Then we  obtain a graded Lie algebra structure $[\cdot,\cdot]^C$ on the complex $\oplus _k\Hom(\wedge^{k}\g\otimes\g,\g)$ via
$$
[P,Q]^C:=\frac{(p+q)!}{p!q!}\alpha([P, Q]^G).
$$
  Note that this graded Lie algebra is exactly the one
given in \cite{ChaLiv} by replacing the left-symmetry by
right-symmetry of the associator.  Moreover, a  bilinear map $\pi:\otimes^2\g\longrightarrow\g$ is a pre-Lie
algebra structure if and only if $\pi$ is a Maurer-Cartan element
($[\pi,\pi]^C=0$) of the graded Lie algebra $(\oplus
_k\Hom(\wedge^{k}\g\otimes\g,\g),[\cdot,\cdot]^C)$ and the
coboundary operator of a pre-Lie algebra can be defined by
$\delta(f)=(-1)^{|f|}[\pi,f]^C$.

  Furthermore, we study deformations of pre-Lie
algebras in terms of the aforementioned graded Lie bracket. In
particular, it leads to the introduction of the notion of a
Nijenhuis operator on a pre-Lie algebra which generates a trivial
deformation of the pre-Lie algebra. It is also motivated by the
similar study for Nijenhuis operators on a Lie algebra which play
an important role in the study of integrability of nonlinear
evolution equations \cite{Dorfman1993}.

  There are close relationships between Nijenhuis operators
and certain known operators on pre-Lie algebras such as
$\huaO$-operators and Rota-Baxter operators, whereas the former
were introduced in \cite{BaiO} as the generalization of the
$S$-equation  in a pre-Lie algebra which is an
analogue of the classical Yang-Baxter equation  in a
Lie algebra  given in \cite{Left-symmetric bialgebras} and the
latter were studied in \cite{AB} and \cite{LHB} with interesting
motivations. On one hand, we illustrate these relationships
explicitly and hence we get some examples of Nijenhuis operators
from the study of $\huaO$-operators and Rota-Baxter operators on
pre-Lie algebras. On the other hand, we find that Nijenhuis
operators play an essential role in certain compatible structures
in terms of $\huaO$-operators. Roughly speaking, a Nijenhuis
operator ``connects'' two $\huaO$-operators on a pre-Lie algebra
whose any linear combination is still an $\mathcal O$-operator
(they are called ``compatible'' $\huaO$-operators)
 in some sense.
We also introduce the notion of compatible L-dendriform algebras
as the naturally induced algebraic structures since the notion of an
L-dendriform algebra was introduced in \cite{BaiO} as the
algebraic structure behind an $\huaO$-operator on a pre-Lie
algebra.

   In particular, there is an interesting geometric
interpretation when we consider the case of the dual
representation of the regular representation of a pre-Lie algebra
for the above approach. In fact, the notion of a pseudo-Hessian
pre-Lie algebra was introduced in \cite{NiBai}, which is the
  algebraic
counterpart of a Lie group with a left-invariant Hessian
structure \cite{shima}. By adding compatibility conditions between a
pseudo-Hessian structure and a Nijenhuis structure on a pre-Lie
algebra, we introduce the notion of a pseudo-Hessian-Nijenhuis
structure. Such a structure has many good properties such as
\begin{enumerate}
\item By a pseudo-Hessian-Nijenhuis structure, we   obtain a sequence of pseudo-Hessian structures and
pseudo-Hessian-Nijenhuis structures. \item The properties that a
pseudo-Hessian-Nijenhuis structure enjoys are totally parallel to
 the ones  of a symplectic-Nijenhuis structure on a Lie algebra
\cite{Kos}. \item The notion of a symmetric $\frks$-matrix in a
pre-Lie algebra was introduced in \cite{Left-symmetric bialgebras}
as a symmetric solution of the $S$-equation which was also interpreted
as an $\huaO$-operator associated to the dual representation of
the regular representation in \cite{BaiO}. Since the inverse of an
$\frks$-matrix is a pseudo-Hessian structure, we construct the
correspondence between pseudo-Hessian-Nijenhuis structures  and
the compatible invertible $\frks$-matrices on a pre-Lie algebra.
\end{enumerate}

 Furthermore, there is another application of Nijenhuis
operators on pre-Lie algebras in geometry. Explicitly, we
introduce the notion of a para-complex structure on a pre-Lie
algebra in which the Nijenhuis condition is exactly the
integrability condition. We show that a para-complex structure on
a pre-Lie algebra gives rise to two transversal pre-Lie
subalgebras. We mainly consider two kinds of structures which
combine para-complex structures and certain symmetric or
skew-symmetric non-degenerate bilinear forms satisfying some
additional conditions. In the skew-symmetric cases, we introduce
the notion of a para-complex quadratic pre-Lie algebra as a
combination of a para-complex structure and a quadratic structure
on a pre-Lie algebra, whereas a quadratic pre-Lie algebra is a
pre-Lie algebra together with a skew-symmetric nondegenerate
invariant bilinear form,  which is the underlying structure of a
symplectic Lie algebra. Moreover,  we show that a para-complex
quadratic pre-Lie algebra is exactly the underlying structure of a
para-K\"{a}hler Lie algebra:
  $$
\xymatrix{
 \mbox{para-K\"{a}hler Lie algebra~}\ar@<1ex>[rr]^{x\cdot_\g y={\omega^\sharp}^{-1}\ad^*_x\omega^\sharp(y)\qquad\qquad}
                && \mbox{para-complex~quadratic~pre-Lie~algebra}  \ar[ll]^{\mbox{commutator}\qquad\qquad}
                }
$$
Para-K\"{a}hler Lie algebras are the algebraic counterpart of
para-K\"{a}hler manifolds \cite{pKm,surveypc} which are widely
studied recently  \cite{Left-symmetric
bialgebras,ben1,ben2,Cal,Cal1}. On the other hand, in the
symmetric cases, we introduce the notion of a paracomplex
pseudo-Hessian structure on a pre-Lie algebra as a combination of
a para-complex structure and a pseudo-Hessian structure on a
pre-Lie algebra, which is equivalent to a pseudo-Hessian pre-Lie
algebra with a pair of isotropic transversal pre-Lie subalgebras
(such a structure is also called a Manin triple of pre-Lie
algebras associated to  a nondegenerate symmetric 2-cocycle in
\cite{NiBai}).

%Pre-Lie algebras are related with many important structures. Given
%a linear map $r:\g\longrightarrow\g$ that satisfies the operator
%form of the classical Yang-Baxter equation in a Lie algebra
%$(\g,[\cdot,\cdot]_\g)$, one can obtain a pre-Lie algebra
%structure $\cdot^r$ on $\g$ defined by $x\cdot^ry=[r(x),y]_\g$.
%For any vector space $\g$ together with an inner produce
%$\langle\cdot,\cdot\rangle$, one can obtain a pre-Lie algebra
%structure $\cdot^a$ on $\g$ defined by $x\cdot^ay=\langle
%x,y\rangle a+\langle x,a\rangle y$ for all $a\in\g$. This pre-Lie
%algebra arose from the study of integrable Burgers equation. Given
%a Rota-Baxter operator $R$ of weight 1 on an associative algebra
%$(\g,*)$, there is a pre-Lie algebra structure $\cdot^R$ defined
%by $x\cdot^R y=R(x)*y-y*R(x)-x*y$. Given a derivation $D$ on a
%commutative associative algebra $(\g,*)$, there is a pre-Lie
%algebra structure $\cdot^s$ defined by $x\cdot^s y=x*D(y)+s(x*y)$
%for all $s\in\mathbb \K$. Finally, we study Nijenhuis operators on the above mentioned pre-Lie algebras.

 The paper is organized as follows. In Section 2, we
review some basic properties of pre-Lie algebras like
representations, coboundary operators, $\frks$-matrices and the
relationship with the Gerstenhaber bracket for associative
algebras. In Section 3, we construct a graded Lie algebra  that are used to characterize pre-Lie algebras.
In Section 4, we introduce the notion of a Nijenhuis operator on a
pre-Lie algebra and show that it generates a trivial deformation
of this pre-Lie algebra. In Section 5, we study the relationships
between $\mathcal O$-operators, Rota-Baxter operators and
Nijenhuis operators on a pre-Lie algebra. In particular, we give
the relationships between Nijenhuis operators and compatible
$\huaO$-operators and compatible L-dendriform algebras in terms of
$\mathcal O$-operators on pre-Lie algebras. In Section 6, we
introduce the notion of a pseudo-Hessian-Nijenhuis structure on a
pre-Lie algebra and study its relation with compatible
$\frks$-matrices. In Section 7, we introduce the notion of a
para-complex structure on a pre-Lie algebra and then introduce and
study para-complex quadratic pre-Lie algebras and para-complex
pseudo-Hessian pre-Lie algebras. In Section 8, we study some
examples of Nijenhuis operators on pre-Lie algebras, whereas the
pre-Lie algebras arose from  the operator forms of the classical
Yang-Baxter equation in Lie algebras, integrable Burgers equation,
Rota-Baxter operators on associative algebras and derivations on
commutative associative algebras respectively.

%\vspace{2mm}
 %\noindent {\bf Acknowledgement:}

 In this paper, we work over an algebraically closed field $\K$ of characteristic 0 and all the vector spaces are over $\K$.  We only work on finite dimensional vector spaces.
\section{Preliminaries}

\subsection{Pre-Lie algebras, representations and cohomologies}
\begin{defi}  A {\bf pre-Lie algebra} is a pair $(\g,\cdot_\g)$, where $\g$ is a vector space and  $\cdot_\g:\g\otimes \g\longrightarrow \g$ is a bilinear multiplication
satisfying that for all $x,y,z\in \g$, the associator
$(x,y,z)=(x\cdot_\g y)\cdot_\g z-x\cdot_\g(y\cdot_\g z)$ is symmetric in $x,y$,
i.e.
$$(x,y,z)=(y,x,z),\;\;{\rm or}\;\;{\rm
equivalently,}\;\;(x\cdot_\g y)\cdot_\g z-x\cdot_\g(y\cdot_\g z)=(y\cdot_\g x)\cdot_\g
z-y\cdot_\g(x\cdot_\g z).$$
\end{defi}

Let $(\g,\cdot_\g)$ be a pre-Lie algebra. The commutator $
[x,y]^c=x\cdot_\g y-y\cdot_\g x$ defines a Lie algebra structure
on $\g$, which is called the {\bf sub-adjacent Lie algebra} of
$(\g,\cdot_\g)$ and denoted by $\g^c$. Furthermore,
$L:\g\longrightarrow \gl(\g)$ with $x\rightarrow L_x$, where
$L_xy=x\cdot_\g y$, for all $x,y\in \g$, gives a representation of
the Lie algebra $\g^c$ on $\g$. See \cite{Pre-lie algebra in
geometry} for more details.

\begin{defi}
Let $(\g,\cdot_\g)$ be a pre-Lie algebra and $V$  a vector
space. A {\bf representation} of $\g$ on $V$ consists of a pair
$(\rho,\mu)$, where $\rho:\g\longrightarrow \gl(V)$ is a representation
of the Lie algebra $\g^c$ on $V $ and $\mu:\g\longrightarrow \gl(V)$ is a linear
map satisfying \begin{eqnarray}\label{representation condition 2}
 \rho(x)\mu(y)u-\mu(y)\rho(x)u=\mu(x\cdot_\g y)u-\mu(y)\mu(x)u, \quad \forall~x,y\in \g,~ u\in V.
\end{eqnarray}
\end{defi}

Usually, we denote a representation by $(V;\rho,\mu)$. It is
obvious that $(  \K  ;\rho=0,\mu=0)$ is a representation, which we
call the {\bf trivial representation}. Let $R:\g\rightarrow
\gl(\g)$ be a linear map with $x\longrightarrow R_x$, where the
linear map $R_x:\g\longrightarrow\g$  is defined by
$R_x(y)=y\cdot_\g x,$ for all $x, y\in \g$. Then
$(\g;\rho=L,\mu=R)$ is also a representation, which we call the
{\bf regular representation}. Define two linear maps $L^*,R^*:\g\longrightarrow
\gl(\g^*)$   with $x\longrightarrow L^*_x$ and
$x\longrightarrow R^*_x$ respectively (for all $x\in \g$)
by
\begin{equation}
\langle L_x^*(\xi),y\rangle=-\langle \xi, x\cdot y\rangle, \;\;
\langle R_x^*(\xi),y\rangle=-\langle \xi, y\cdot x\rangle, \;\;
\forall x, y\in \g, \xi\in \g^*.
\end{equation}
Then $(\g^*;\rho={\rm ad}^*=L^*-R^*, \mu=-R^*)$ is a
representation of $(\g,\cdot_\g)$. In fact, it is the dual
representation of the regular representation $(\g;L,R)$.

The cohomology complex for a pre-Lie algebra $(\g,\cdot_\g)$ with a representation $(V;\rho,\mu)$ is given as follows (\cite{cohomology of pre-Lie}).
The set of $n$-cochains is given by
$\Hom(\wedge^{n-1}\g\otimes \g,V),\
n\geq 1.$  For all $\phi\in \Hom(\wedge^{n-1}\g\otimes \g,V)$, the coboundary operator $\dM:\Hom(\wedge^{n-1}\g\otimes \g,V)\longrightarrow \Hom(\wedge^{n}\g\otimes \g,V)$ is given by
 \begin{eqnarray}\label{eq:pre-Lie cohomology}
 \nonumber&&\dM\phi(x_1, \cdots,x_{n+1})\\
 \nonumber&=&\sum_{i=1}^{n}(-1)^{i+1}\rho(x_i)\phi(x_1, \cdots,\hat{x_i},\cdots,x_{n+1})\\
\nonumber &&+\sum_{i=1}^{n}(-1)^{i+1}\mu(x_{n+1})\phi(x_1, \cdots,\hat{x_i},\cdots,x_n,x_i)\\
 \nonumber&&-\sum_{i=1}^{n}(-1)^{i+1}\phi(x_1, \cdots,\hat{x_i},\cdots,x_n,x_i\cdot_\g x_{n+1})\\
\label{eq:cobold} &&+\sum_{1\leq i<j\leq n}(-1)^{i+j}\phi([x_i,x_j]^c,x_1,\cdots,\hat{x_i},\cdots,\hat{x_j},\cdots,x_{n+1}),
\end{eqnarray}
for all $x_i\in \g,~i=1,\cdots,n+1$. In particular, we use the
symbol $\dM^T$ to refer the coboundary operator   associated to
the trivial representation and $\dr$ to refer the coboundary
operator  associated to the regular representation. We denote the
$n$-th cohomology group for the coboundary operator $\dr$  by
$H_{\rm reg}^n(\g,\g)$ and $H_{\rm reg}(\g,\g)=\oplus_{n}H_{\rm reg}^n(\g,\g)$.

 \begin{defi}{\rm(\cite{NiBai})}
 A {\bf pseudo-Hessian  structure} on a pre-Lie algebra $(\g,\cdot_\g)$  is a symmetric nondegenerate $2$-cocycle $\frkB\in\Sym^2(\g^*)$, i.e. $\dM^T\frkB=0$. More precisely,
$$\frkB(x\cdot_\g y,z)-\frkB(x,y\cdot_\g z)=\frkB(y\cdot_\g x,z)-\frkB(y,x\cdot_\g z),\quad \forall~x,y,z\in\g.$$
A pre-Lie algebra equipped with a pseudo-Hessian   structure  is called  a {\bf pseudo-Hessian pre-Lie algebra},
and denoted by $(\g,\cdot_\g,\frkB)$.
\end{defi}

\subsection{$\frks$-matrices in
pre-Lie algebras}

\emptycomment{
\begin{defi}{\rm(\cite{Left-symmetric bialgebras})}\label{defi:pre-Lie bialgebra}
 A {\bf pre-Lie bialgebra}  is a pair of pre-Lie algebras $\g$ and $\g^*$  such that
\begin{itemize}
\item[$\rm(a)$]$\alpha$ is a $1$-cocycle on $\g^c$ with the
coefficients in the representation $L\otimes1+1\otimes
\ad$;
\item[$\rm(b)$]$\beta$ is a $1$-cocycle on ${\g^*}^c$
with the coefficients in the representation
$L\otimes1+1\otimes \ad$,
\end{itemize}
where $\alpha:\g\rightarrow \g\otimes\g$ and $\beta:\g^*\rightarrow
\g^*\otimes\g^*$ are linear maps such that
$\alpha^*:\g^*\otimes\g^*\longrightarrow \g^*$ is the pre-Lie
algebra structure on $\g^*$ and
$\beta^*:\g\otimes\g\longrightarrow\g$ is the pre-Lie algebra
structure on $\g$, $\ad=L-R$   denotes the adjoint representation.
\end{defi}

\begin{defi}{\rm(\cite{Left-symmetric bialgebras})}
A pre-Lie bialgebra $(\g,\g^*)$  is said to be {\bf coboundary} if there exists an $r\in\g\otimes\g$ such that
\begin{equation}\label{eq:coboundary}\alpha(x)=(L_x\otimes1+1\otimes\ad_x)r,\quad \forall~x\in\g.\end{equation}
\end{defi}
}

Let $\g$ be a vector space. %Define$$\Sym^2(\g)=\{r\in\g\otimes\g\mid r(\xi,\eta)=r(\eta,\xi),\quad \forall~\xi,\eta\in\g^*\}.$$
For any $r\in\Sym^2(\g)$, the linear map $r^\sharp:\g^*\longrightarrow \g$ is given by
$\langle r^\sharp(\xi),\eta\rangle=r(\xi,\eta),$ for all $~\xi,\eta\in\g^*$. We say that $r\in\Sym^2(\g)$ is invertible, if the linear map $r^\sharp$ is an isomorphism.

Let $(\g,\cdot_\g)$ be a pre-Lie algebra and $r\in\Sym^2(\g)$. We introduce $\llbracket r,r\rrbracket\in\wedge^2\g\otimes\g$ as follows
\begin{equation}\label{S-equation1}
\llbracket r,r\rrbracket(\xi,\eta,\zeta)=-\langle\xi,r^\sharp(\eta)\cdot_\g r^\sharp(\zeta)\rangle+\langle\eta,r^\sharp(\xi)\cdot_\g r^\sharp(\zeta)\rangle+\langle\zeta,[r^\sharp(\xi),r^\sharp(\eta)]^c\rangle,\quad\forall~\xi,\eta,\zeta\in\g^*.
\end{equation}

\begin{defi}{\rm(\cite{Left-symmetric bialgebras})} Let
$(\g,\cdot_\g)$ be a pre-Lie algebra.  If $r\in\Sym^2(\g)$ and
satisfies $\llbracket r,r\rrbracket=0$, then $r$ is called an {\bf
$\frks$-matrix} in $(\g,\cdot_\g)$.
\end{defi}

\begin{rmk}
An $\frks$-matrix in a pre-Lie algebra $(\g, \cdot_\g)$ is a
solution of the $S$-equation in $(\g, \cdot_\g)$ which is an
analogue of the classical Yang-Baxter equation in a Lie algebra.
It plays an important role in the theory of pre-Lie bialgebras,
see \cite{Left-symmetric bialgebras} for more details.
\end{rmk}

Let $(\frkg,\cdot_\frkg)$ be a pre-Lie algebra and $r$ an $\frks$-matrix. Then $(\g^*,\cdot_r)$ is a pre-Lie algebra, where the multiplication $\cdot_r:\otimes ^2\g^*\longrightarrow\g^*$ is  given by
\begin{equation}\label{eq:pre-bia}
\xi\cdot_{r}\eta=\ad^*_{r^\sharp(\xi)}\eta-R^*_{r^\sharp(\eta)}\xi,\quad \forall~\xi,\eta\in\g^*.
\end{equation}

\begin{pro}\label{pro:morphism}
With the above notations,
for all $\xi,\eta\in\g^*$, we have
\begin{equation}
r^\sharp(\xi\cdot_r \eta)-r^\sharp(\xi)\cdot_\g r^\sharp(\eta)=0.
\end{equation}
\end{pro}

\begin{pro}\label{pro:Hessian1}{\rm(\cite{Left-symmetric
bialgebras})} Let $(\frkg,\cdot_\frkg)$ be a pre-Lie algebra and
$r$ an invertible $\frks$-matrix. Define $\frkB\in \Sym^2(\g^*)$
by
\begin{equation}\label{eq:rB}
\frkB(x,y)=\langle x,(r^\sharp)^{-1}(y)\rangle,\quad \forall~x,y\in\g.
\end{equation}
Then $(\g,\cdot_\g,\frkB)$ is a pseudo-Hessian pre-Lie algebra.
\end{pro}

Now we extend the equation $\eqref{S-equation1}$ to any two elements $r_1,r_2\in\Sym^2(\g)$ as follows:
\begin{eqnarray}
\llbracket r_1,r_2\rrbracket(\xi,\eta,\zeta)&=&\half\big(-\langle\xi,r^\sharp_1(\eta)\cdot_\g r^\sharp_2(\zeta)+r^\sharp_2(\eta)\cdot_\g r^\sharp_1(\zeta)\rangle+\langle\eta,r^\sharp_1(\xi)\cdot_\g r^\sharp_2(\zeta)+r^\sharp_2(\xi)\cdot_\g r^\sharp_1(\zeta)\rangle\nonumber\\
&&+\langle\zeta,[r^\sharp_1(\xi),r^\sharp_2(\eta)]^c+[r^\sharp_2(\xi),r^\sharp_1(\eta)]^c\rangle\big).\label{eq:S-equation2}
\end{eqnarray}
It is easy to see that $\llbracket r_1,r_2\rrbracket\in \wedge^2\g\otimes\g$ and
\begin{eqnarray*}
\llbracket r_1,r_2\rrbracket&=&\llbracket r_2,r_1\rrbracket,\\
\llbracket r_1,\lambda r_2\rrbracket&=&\lambda\llbracket r_1,r_2\rrbracket,\quad \forall \lambda\in {\mathbb K}, \\
\llbracket r_1+r_2,r_1+r_2\rrbracket&=&\llbracket r_1,r_1\rrbracket+2\llbracket r_1,r_2\rrbracket+\llbracket r_2,r_2\rrbracket.
\end{eqnarray*}

\begin{defi}\label{defi:cs}
Let $r_1$ and $r_2$ be two $\frks$-matrices. $r_1$ and $r_2$ are called {\bf compatible} if $\llbracket r_1,r_2\rrbracket=0.$
\end{defi}
  It is obvious that $r_1$ and $r_2$ are compatible if and only
if for all  $t_1,t_2\in \mathbb K$, $t_1r_1+t_2
r_2$ are $\frks$-matrices.

\subsection{The Gerstenhaber
bracket, (graded) pre-Lie algebras and associative algebras}

Let $\g$ be a vector space. Denote by
$A^p(\g,\g)=\Hom(\otimes^{p+1}\g,\g)$ and set
$A(\g,\g)=\oplus_{p\in\Nat}A^p(\g,\g).$ We assume that the degree
of an element in $A^p(\g,\g)$ is $p$. Let $P\in A^p(\g,\g),Q\in
A^q(\g,\g),\quad p,q\geq0$ and $x_i\in \g,\quad
i=1,2,\ldots,p+q+1$. Then we have
\begin{thm}\rm(\cite{Gerstenhaber1})
The graded vector space $A(\g,\g)$ equipped with the graded commutator bracket
\begin{equation}
[P,Q]^G=P\circ Q-(-1)^{pq}Q\circ P,
\end{equation}
is a graded Lie algebra, where $P\circ Q\in A^{p+q}(\g,\g)$ is defined by
\begin{eqnarray*}
P\circ Q(x_1,x_2,\ldots,x_{p+q+1})=\sum_{i=1}^{p+1}(-1)^{(i-1)q}P(x_1,\ldots,x_{i-1},Q(x_{i},\ldots,x_{i+q}),x_{i+q+1},\ldots,x_{p+q+1}).
\end{eqnarray*}
Furthermore, $\pi\in\Hom(\otimes^2\g,\g)$ is an associative algebra if and only if $[\pi,\pi]^G=0$. This bracket is usually called the {\bf Gerstenhaber bracket}.
\end{thm}

Note that $(A(\g,\g),\circ)$ is a graded pre-Lie algebra,
that is,
$$(x,y,z)=(-1)^{pq}(y,x,z),\quad \forall x\in A^p(\g,\g), y\in
A^q(\g,\g), z\in A^r(\g,\g).$$

Assume that $(\g,\ast)$ is an associative algebra and set
$\pi(x,y)=x\ast y$. Then the Hochschild coboundary operator
$\dM^H:A^{p-1}(\g,\g)\longrightarrow A^p(\g,\g)$ which is given by
\begin{eqnarray*}
\dM^{H}P(x_1,\ldots,x_{p+1})&=&x_1\ast P(x_2,\ldots,x_{p+1})+\sum_{i=1}^p(-1)^iP(x_1,\ldots,x_{i-1},x_i\ast x_{i+1},\ldots,x_{p+1})\\
&&+(-1)^{p+1}P(x_1,\ldots,x_p)\ast x_{p+1}
\end{eqnarray*}
can be rewritten in a simple form $\dM^H(P)=(-1)^{p-1}[\pi,P]^G$ for $P\in A^{p-1}(\g,\g)$.

\section{Generalized Gerstenhaber
bracket and pre-Lie algebras}

  With the notations in Subsection 2.3, define  an {\bf alternator} $\alpha:A(\g,\g)\longrightarrow A(\g,\g)$ by
\begin{equation}
\alpha(P)(x_1,\ldots,x_{p+1})=\frac{1}{p!}\sum_{\sigma\in S_p}\sgn(\sigma)P(x_{\sigma(1)},\ldots,x_{\sigma(p)},x_{p+1}),\quad \forall~P\in A^p(\g,\g).
\end{equation}
By direct calculation, we have
$\alpha^2=\alpha,$
which means that the map $\alpha$ is a projection.

Denote by $C^p(\g,\g)=\Hom(\wedge^p\g\otimes  \g,\g)=\alpha(A^p(\g,\g)) $ and  $C(\g,\g)=\bigoplus_{p\in\Nat}C^p(\g,\g)$.
For the graded subspace ${C}(\g,\g)$ of ${A}(\g,\g)$, define the binary operation $\diamond:{C}(\g,\g)\times {C}(\g,\g)\longrightarrow {C}(\g,\g)$ by
 $$
 P\diamond Q=\frac{(p+q)!}{p!q!}\alpha(P\circ Q),\quad \forall~P\in C^p(\g,\g),Q\in C^q(\g,\g).
 $$
 More precisely, we have
\begin{eqnarray*}
&&P\diamond Q(x_1,\ldots,x_{p+q+1})\\
&=&\sum_{\sigma\in\Delta}\sgn(\sigma)P(Q(x_{\sigma(1)},\ldots,x_{\sigma(q)},x_{\sigma(q+1)}),x_{\sigma(q+2)},\ldots,x_{\sigma(p+q)},x_{p+q+1})\\
&&+(-1)^{pq}\sum_{\sigma\in (p,q)-unshuffles}\sgn(\sigma)P(x_{\sigma(1)},\ldots,x_{\sigma(p)},Q(x_{\sigma(p+1)},\ldots,x_{\sigma(p+q)},x_{p+q+1})),
\end{eqnarray*}
where $\sigma\in\Delta$ means $\sigma(1)<\ldots<\sigma(q),\sigma(q+2)<\ldots<\sigma(p+q).$

Denote by $[\cdot,\cdot]^C$ the corresponding commutator bracket given by
\begin{eqnarray}\label{eq:Nij-preLie}
{[P,Q]^C}&=&P\diamond Q-(-1)^{pq}Q\diamond P=\frac{(p+q)!}{p!q!}\alpha([P, Q]^G).
\end{eqnarray}

\begin{thm}\label{thm:lsymNR}
With the above notations, we have
\begin{itemize}
\item[\rm(i)] $(C(\g,\g),{[\cdot,\cdot]^C})$ is a graded Lie algebra;
\item[\rm(ii)]$\pi\in\Hom(\otimes^2\g,\g)$ defines a pre-Lie algebra if and only if $[\pi,\pi]^C=0.$
\end{itemize}
\end{thm}
\pf (i)  It is obvious that $[\cdot,\cdot]^C$ is graded commutative.
By a long but straightforward computation, we can show that for $P\in A^p(\g,\g),Q\in A^q(\g,\g)$,
 $$\alpha(P\circ Q)=\alpha(\alpha(P)\circ \alpha(Q)).$$
 Thus for $P\in C^p(\g,\g),Q\in C^q(\g,\g),R\in C^r(\g,\g)$, by \eqref{eq:Nij-preLie}, we have
\begin{eqnarray*}
\alpha([[P,Q]^G,R]^G)&=&\alpha\big([P,Q]^G\circ R-(-1)^{(p+q)r}R\circ [P,Q]^G\big)\\
&=&\alpha\big(\alpha([P,Q]^G)\circ\alpha( R)-(-1)^{(p+q)r}\alpha(R)\circ \alpha([P,Q]^G)\big)\\
&=&\frac{(p+q)!r!}{(p+q+r)!}[\alpha([P,Q]^G),\alpha(R)]^C\\
&=&\frac{p!q!r!}{(p+q+r)!}[[P,Q]^C,R]^C.
\end{eqnarray*}
Thus, the graded Jacobi identity for the bracket $[\cdot,\cdot]^C$ follows directly from the graded Jacobi identity for the bracket $[\cdot,\cdot]^G$.

(ii) By direct calculation, we have
\begin{eqnarray*}
\half[\pi,\pi]^C(x_1,x_2,x_3)&=&\pi\diamond\pi(x_1,x_2,x_3)\\
&=&\pi(\pi(x_1,x_2),x_3)-\pi(\pi(x_2,x_1),x_3)-\pi(x_1,\pi(x_2,x_3))+\pi(x_2,\pi(x_1,x_3)),
\end{eqnarray*}
which implies that $\pi$ defines a pre-Lie algebra if and only if $[\pi,\pi]^C=0$.\qed

\begin{rmk} By a similar proof as that of the above conclusion
(i), $(C(\g,\g),\diamond)$ is also a graded pre-Lie algebra.
The graded Lie algebra given above is exactly the graded Lie
algebra given in \cite{ChaLiv} using the operad theory if the
left-symmetry of the associator is replaced by the right-symmetry. Here
the graded Lie algebra is given directly and explicitly through
the skew-symmetrization of the Gerstenhaber bracket.
\end{rmk}

Let $(\g,\pi)$ be a pre-Lie algebra. By Theorem \ref{thm:lsymNR}, we have $[\pi,\pi]^C=0$. Because of the graded Jacobi identity, we get a coboundary operator ${\delta}:C^n(\g,\g)\longrightarrow C^{n+1}(\g,\g)$ defined by
\begin{equation}\label{eq:cob}
{\delta}(\phi)=(-1)^n[\pi,\phi]^C,\quad\forall~\phi\in C^n(\g,\g).
\end{equation}

By straightforward computation, we have
\begin{pro}
For all $\phi\in C^{n-1}(\g,\g)$, we have
\begin{eqnarray}\label{eq:pre-Lie cohomology2}
\nonumber{\delta}\phi(x_1,\cdots,x_{n+1})&=&\sum_{i=1}^{n}(-1)^{i+1}\pi(x_i,\phi(x_1,\cdots,\hat{x_i},\cdots,x_{n+1}))\\
&&+\sum_{i=1}^{n}(-1)^{i+1}\pi(\phi(x_1,\cdots,\hat{x_i},\cdots,x_n,x_i), x_{n+1})\\
\nonumber &&-\sum_{i=1}^{n}(-1)^{i+1}\phi(x_1,\cdots,\hat{x_i},\cdots,x_n,\pi(x_i, x_{n+1}))\\
\nonumber &&+\sum_{1\leq i<j\leq n}(-1)^{i+j}\phi(\pi(x_i, x_j)-\pi(x_j, x_i),x_1,\cdots,\hat{x_i},\cdots,\hat{x_j},\cdots,x_{n+1}).
\end{eqnarray}
Thus, the coboundary operator given by \eqref{eq:cob} is
  exactly  the one given by \eqref{eq:cobold} for the regular
representation $(L,R)$, i.e. $\delta=\dr$.
\end{pro}

\emptycomment{
The relation of the cohomology group of pre-Lie algebras and the sub-adjacent Lie algebras can be described by the following proposition.
\begin{pro}
Let $(V,\cdot)$ be a pre-Lie algebras. There is a natural map $\bar{\alpha}:C^p(V,V)\longrightarrow \bar{C}^p(V,V)$ inducing a map (with the same symbol) $\bar{\alpha}:H_{\delta}(V,V)\longrightarrow H_{\dM^{CE}}(V,V)$, where the coboundary operators $\delta$ and $\dM^{CE}$ are corresponding to the pre-Lie algebras and its sub-adjacent Lie algebra respectively.
\end{pro}
\pf By direct calculation, we have
$$\bar{\alpha}(P\diamond Q)=c(\bar{\alpha}(P)\ast \bar{\alpha}(Q)),\quad P\in C^p(V,V),Q\in C^q(V,V),$$
where $c$ is a constant dependence on $p$ and $q$.

Let $\pi$ denote the operation $\cdot$ and $\bar{\pi}$ the sub-adjacent Lie algebra structure. Then we have
$$\bar{\alpha}[\pi,P]^N=c[\bar{\pi},\bar{\alpha}(P)]^{NR}.$$
 The proposition follows immediately.\qed
 }

Let $V$ and $W$ be two vector spaces. A linear map $f:V\longrightarrow\gl(W)$ can induce two maps $\bar{f}:\otimes^2(V\oplus W)\longrightarrow V\oplus W$ and $\tilde{f}:\otimes^2(V\oplus W)\longrightarrow V\oplus W$ by
\begin{eqnarray*}
\bar{f}(v_1+w_1,v_2+w_2)&=&f(v_1)(w_2),\\
\tilde{f}(v_1+w_1,v_2+w_2)&=&f(v_2)(w_1),\quad\forall~ v_1,v_2\in V,w_1,w_2\in W.
\end{eqnarray*}

Consider a representation of a pre-Lie algebra, we have
\begin{pro}
Let $(\g,\cdot_\g)$ be a pre-Lie algebra and $V$  a vector space. A pair of linear maps $\rho,\mu:\g\longrightarrow\gl(V)$ is a representation of the pre-Lie algebra $\g$ on $V$ if and only if
\begin{equation}
[\pi+\bar{\rho}+\tilde{\mu},\pi+\bar{\rho}+\tilde{\mu}]^C=0.
\end{equation}
\end{pro}
\pf A pair of linear maps $\rho,\mu:\g\longrightarrow\gl(V)$ is a representation of the pre-Lie algebra $\g$ on $V$ if and only if $\g\ltimes_{\rho,\mu} V$ is a pre-Lie algebra, where the pre-Lie algebra structure on $\g\ltimes_{\rho,\mu} V$ is given by
\begin{eqnarray}
\label{eq:mulsemi}(x_1+v_1)\star (x_2+v_2)&=&x_1\cdot_\g x_2+\rho(x_1)v_2+\mu(x_2)v_1\\
&=&\nonumber(\pi+\bar{\rho}+\tilde{\mu})(x_1+v_1,x_2+v_2),\quad \forall~x_1,x_2\in\g,v_1,v_2\in V.
\end{eqnarray}
Thus, by Theorem \ref{thm:lsymNR}, the conclusion follows immediately.\qed

\section{Nijenhuis operators on pre-Lie algebras}\label{sec:Nijenhuis}

Let $(\g,\cdot_\g)$ be a pre-Lie algebra and $N:\g\longrightarrow \g$  a linear map. We also use $\pi:\otimes^2\g\longrightarrow \g$ to denote the multiplication in the pre-Lie algebra, i.e. $\pi(x,y)=x\cdot_\g y$. The deformed structure
$$\pi_N:=[\pi,N]^C$$
defines a bilinear operation on $\g$ which we denote by $\cdot_N$. Explicitly,
\begin{equation}
x\cdot_N y=[\pi,N]^C(x,y),\quad \forall~x,y\in \g.
\end{equation}

\begin{lem}
The operation $\cdot_N$   satisfies
 \begin{equation}\label{eq:LANij}
x\cdot_N y=N(x)\cdot_\g y+x\cdot_\g N(y)-N(x\cdot_\g y),\quad \forall x,y\in \g.
\end{equation}
\end{lem}

We use $T_\pi N:\g\otimes\g\longrightarrow\g$ to denote the
Nijenhuis torsion of $N$ defined by\footnote{This definition is
motivated by the Nijenhuis torsion of a Lie algebra given in
\cite[Proposition 2.3]{Kos}.}
\begin{equation}\label{eq:LANij1}
T_\pi N:=\half([\pi,N\diamond N]^C+[N,[\pi,N]^C]^C).
\end{equation}
\begin{lem}\label{lem:morp}
  With the above notations, for all $x,y\in\g$, we have
  \begin{equation}\label{eq:TN}
T_\pi N(x,y)=N(x\cdot_N y)-N(x) \cdot_\g N(y).
\end{equation}
\end{lem}
\pf For all $x,y\in \g$, we have
\begin{eqnarray*}
[\pi,N\diamond N]^C(x,y)=N^2(x)\cdot_\g y+x\cdot_\g N^2(y)-N^2(x\cdot_\g y).
\end{eqnarray*}
On the other hand, we have
\begin{eqnarray*}
[N,[\pi,N]^C]^C(x,y)&=&N(N(x)\cdot_\g y+x\cdot_\g N(y)-N(x\cdot_\g y))\\
&&-N^2(x)\cdot_\g y-N(x)\cdot_\g N(y)+N(N(x)\cdot_\g y)\\
&&-N(x)\cdot_\g N(y)-x\cdot_\g N^2(y)+N(x\cdot_\g N(y)).
\end{eqnarray*}
Thus, we have
\begin{eqnarray*}
([\pi,N\diamond N]^C+[N,[\pi,N]^C]^C)(x,y)=2\Big(N\big(N(x)\cdot_\g y+x\cdot_\g N(y)-N(x\cdot_\g y)\big)-N(x)\cdot_\g N(y)\Big),
\end{eqnarray*}
which implies that \eqref{eq:TN} holds. \qed

\begin{thm}\label{thm:LANij}
Let $(\g,\pi)$ be a pre-Lie algebra and $N$ a linear map. Then the deformed structure $\pi_N=[\pi,N]^C$ is a pre-Lie algebra if and only if the Nijenhuis torsion $T_\pi N$ of $N$ is closed, i.e. $\dr(T_\pi N)=0$.
\end{thm}
\pf By the graded Jacobi identity, we have
 \begin{equation}\label{eq:LANij2}
\half([[\pi,N]^C,[\pi,N]^C]^C)=[\pi,T_\pi N]^C,
\end{equation}
which implies that $(\g,\pi_N)$ is a pre-Lie algebra if and only if  $\dr(T_\pi N)=0$.\qed

\begin{rmk}
The deformed pre-Lie algebra structure $\pi_N$ is compatible with $\pi$ in the
sense that $\pi + \pi_N$ is a pre-Lie algebra structure, i.e. $[\pi + \pi_N,\pi + \pi_N]^C=0.$

\end{rmk}
\begin{defi}
Let $(\g,\pi)$ be a pre-Lie algebra. A linear map $N:\g\longrightarrow \g$ is called a {\bf Nijenhuis operator} if $T_\pi N=0$, i.e.
\begin{equation}\label{eq:nij-opera}
N(x)\cdot_\g N(y)=N\big(N(x)\cdot_\g y+x\cdot_\g N(y)-N(x\cdot_\g
y)\big),\;\forall x,y\in\g.
\end{equation}
\end{defi}

By Lemma \ref{lem:morp} and Theorem \ref{thm:LANij}, we have
\begin{pro}\label{pro:LANij}
Let $N:\g\longrightarrow \g$ be a Nijenhuis operator on a pre-Lie algebra $(\g,\pi)$. Then $(\g,\pi_N)$ is also a pre-Lie algebra and $N$ is a morphism from the pre-Lie algebra $(\g,\pi_N)$ to $(\g,\pi)$.
\end{pro}
Recall that a {\bf Nijenhuis operator} $N$ on a Lie algebra $(\g,[\cdot,\cdot]_\g)$ is a linear operator $N:\g\longrightarrow \g$ satisfying
$$N([N(x),y]_\g+[x,N(y)]_\g-N([x,y]_\g))=[N(x),N(y)]_\g,\quad \forall~x,y\in\g.$$
Then $(\g,[\cdot,\cdot]_N)$ is a Lie algebra and $N$ is a morphism
from the Lie algebra $(\g,[\cdot,\cdot]_N)$ to
$(\g,[\cdot,\cdot]_\g)$, where the bracket  $[\cdot,\cdot]_N$ is
defined by
$$[x,y]_N=[N(x),y]_\g+[x,N(y)]_\g-N([x,y]_\g),\;\;\forall x,y\in \g.$$

The following   conclusion  is obvious.
\begin{pro}
Let $N$ be a Nijenhuis operator   on a pre-Lie algebra $(\g,\cdot_\g)$. Then $N$ is also a Nijenhuis operator on the sub-adjacent Lie algebra $(\g,[\cdot,\cdot]^c)$. Furthermore, the commutator bracket of the multiplication $\cdot_N$ is exactly the Lie bracket $[\cdot,\cdot]^c_N$, i.e.
$$
[x,y]^c_N=x\cdot_N y-y\cdot_Nx,\;\;\forall x,y\in \g.
$$
\end{pro}

In the following, we study deformations of pre-Lie algebras using the graded Lie algebra $(C(\g,\g),{[\cdot,\cdot]^C})$ given in Theorem \ref{thm:lsymNR}. Let $(\g,\pi)$ be a pre-Lie algebra, and
$\omega\in \Hom(\otimes^2\g,\g)$. Consider a $t$-parameterized family
of multiplications $\pi_t:\g\otimes \g\longrightarrow \g$ given by
\begin{eqnarray*}
\pi_t(x, y)&=&\pi(x,y)+t\omega(x,y),\;\;\forall x,y\in \g.
\end{eqnarray*}
If  $\g_t=(\g,\pi_t)$ is a pre-Lie algebra for all $t$, we say
that $\omega$ generates a {\bf $1$-parameter infinitesimal
deformation} of $(\g,\pi)$.

Obviously,  $(\g,\pi_t)$ is a deformation of $(\g,\pi)$  if and only if $[\pi_t,\pi_t]^C=0$,
which is equivalent to
 \begin{eqnarray}
[\pi,\omega]^C&=&0,\label{2-closed}\\
{[\omega,\omega]}^C&=&0.\label{omega bracket}
\end{eqnarray}
Eq. $(\ref{2-closed})$ means that $\omega$ is a $2$-cocycle for the pre-Lie algebra $(\g,\pi)$, i.e. $\dr\omega=0$, and Eq. $(\ref{omega bracket})$ means that $(\g,\omega)$ is a pre-Lie algebra.

Two deformations $\g_t=(\g,\pi_t)$ and $\g'_t=(\g,\pi'_t)$  of a
pre-Lie algebra $(\g,\cdot_\g)$, which are generated by $\omega$
and $\omega'$  respectively, are said to be {\bf
equivalent} if there exists a family of pre-Lie algebra
isomorphisms ${\Id}+tN:\g_t\longrightarrow \g'_t$. A deformation
is said to be {\bf trivial} if there exist a family of pre-Lie
algebra isomorphisms ${\Id}+tN:\g_t\longrightarrow (\g,\pi)$.

By straightforward computations, $\g_t$ and $\g'_t$  are
equivalent deformations  if and only if
\begin{eqnarray}
\omega(x,y)-\omega'(x,y)&=&x\cdot_\g N(y)+N(x)\cdot_\g y-N(x\cdot_\g y),\label{2-exact}\\
N\omega(x,y)&=&\omega'(x,N(y))+\omega'(N(x),y)+N(x)\cdot_\g N(y),\label{integral condition 1}\\
\omega'(N(x),N(y))&=&0,\label{eq:con1}
\end{eqnarray}
in which $x\cdot_\g y=\pi(x,y),$ for all $~x,y\in \g.$

Eq. $(\ref{2-exact})$ means that $\omega-\omega'=\dr
N$.  Eq. \eqref{eq:con1} means that $\omega'|_{\Img(N)}=0.$

We summarize the above discussion into the following
 conclusion:
\begin{thm}\label{thm:deformation}
Let $(\g,\pi_t)$ be a deformation of a pre-Lie algebra $(\g,\pi)$
generated by $\omega\in\Hom(\otimes^2\g,\g)$. Then $\omega$ is
closed, i.e. $\dr\omega=0.$ Furthermore, if two deformations
$(\g,\pi_t)$ and $(\g,\pi'_t)$ generated by $\omega $ and
$\omega'$  respectively  are equivalent,  then
$\omega$ and $\omega'$ are in the same cohomology class in $H_{\rm
reg}^2(\g,\g)$.
\end{thm}
Now we consider  trivial deformations of a pre-Lie algebra
$(\g,\cdot_\g)$. Eqs.
$(\ref{2-exact})-(\ref{eq:con1})$ reduce to
\begin{eqnarray}
\omega(x,y)&=&x\cdot_\g N(y)+N(x)\cdot_\g y-N(x\cdot_\g y),\label{Nij1}\\
N(\omega(x,y))&=&N(x)\cdot_\g N(y)\label{Nij2}.
\end{eqnarray}
 It follows from $(\ref{Nij1})$ and $(\ref{Nij2})$ that $N$ is a Nijenhuis operator.

We have seen that a trivial deformation of a pre-Lie algebra gives rise to a Nijenhuis operator. In fact, the converse is also true.

\begin{thm}\label{thm:trivial def}
  Let $N$ be a Nijenhuis operator on a pre-Lie algebra $(\g,\cdot_\g)$. Then a
  deformation of  $(\g,\cdot_\g)$ can be
  obtained by putting
  $$
\omega(x,y)= (\dr N)(x,y),\;\;\forall x,y\in \g.
  $$
  Furthermore, this deformation is trivial.
\end{thm}
\pf Since $\omega$ is exact, $\omega $ is closed naturally, i.e. Eq.
\eqref{2-closed} holds. To see that $\omega$ generates a
deformation, we only need to show that \eqref{omega bracket} holds,
which follows from the Nijenhuis condition \eqref{eq:nij-opera}. We omit the details. At last, it is obvious that
$$
({\Id} +tN)\pi_t(x, y)=\pi(({\Id}+tN) (x),({\Id}+tN)(y)),
$$
which implies that the deformation is trivial. \qed
\vspace{3mm}

As in the Lie algebra case, any polynomial of a Nijenhuis operator is still a Nijenhuis operator.
The following formula can be obtained by straightforward computations.
\begin{lem}\label{lem:Nij1}
Let $N$ be a Nijenhuis operator   on a pre-Lie algebra
$(\g,\cdot_\g)$. Then for arbitrary positive $j,k\in\Nat$,  the following equation holds
\begin{equation}
N^j(x)\cdot_\g N^k(y)-N^k(N^j(x)\cdot_\g y)-N^j(x\cdot_\g N^k(y))+N^{j+k}(x\cdot_\g y)=0,\quad\forall~ x,y\in \g.
\end{equation}
If $N$ is invertible, this formula is valid for arbitrary $j,k\in\mathbb Z$.
\end{lem}

\begin{thm}
Let $N$ be a Nijenhuis operator
on a pre-Lie algebra $(\g,\cdot_\g)$. Then for any polynomial
$P(z)=\sum_{i=0}^nc_iz^i$ the operator $P(N)$ is also a Nijenhuis
operator. If $N$ is invertible,  then $Q(N)$ is also a Nijenhuis
operator for any polynomial $Q(z)=\sum_{i=-m}^nc_iz^i$.
\end{thm}
\pf For $x,y\in \g$, by Lemma \ref{lem:Nij1}, we have
\begin{eqnarray*}
&&P(N)(x)\cdot_\g P(N)(y)-P(N)(P(N)(x)\cdot_\g y)-P(N)(x\cdot_\g P(N)(y))+P^2(N)(x\cdot_\g y)\\
&=&\sum_{i,j=0}^nc_jc_k\Big(N^j(x)\cdot_\g N^k(y)-N^k(N^j(x)\cdot_\g y)-N^j(x\cdot_\g N^k(y))+N^{j+k}(x\cdot_\g y)\Big)=0,
\end{eqnarray*}
which implies that $P(N)$ is a Nijenhuis operator. The second statement is valid for similar reasons.\qed

\section{$\huaO$-operators, Rota-Baxter operators and Nijenhuis
operators}

In the first subsection, we illustrate the relationships between
$\huaO$-operators on pre-Lie algebras  introduced in \cite{BaiO},
Rota-Baxter operators on pre-Lie algebras  and Nijenhuis operators
on pre-Lie algebras. In the second subsection, we give the
relationships between Nijenhuis operators and certain compatible
structures like compatible $\huaO$-operators and compatible
L-dendriform algebras  in terms of $\huaO$-operators on
pre-Lie algebras.

\subsection{Relationships between $\huaO$-operators, Rota-Baxter operators and Nijenhuis
operators}

Recall that an {\bf $\huaO$-operator} on a pre-Lie algebra
$(\g,\cdot_\g)$ associated to a representation $(V;\rho,\mu)$ is a
linear map $T:V\longrightarrow \g$ satisfying
\begin{equation}
  T(u)\cdot_\g T(v)=T\Big(\rho(T(u))(v)+\mu(T(v))(u)\Big),\quad \forall u,v\in V.
\end{equation}
By Proposition \ref{pro:morphism}, if $r$ is an $\frks$-matrix,
then $r^\sharp$ is an $\huaO$-operator on the pre-Lie algebra
$(\g,\cdot_\g)$ associated to   the  representation
$(\g^*;\ad^*,-R^*)$.

\begin{defi}
Let $(\g,\cdot_\g)$ be a pre-Lie algebra and $\huaR:\g\longrightarrow
\g$ a   linear operator. If $\huaR$ satisfies \begin{equation}
\huaR(x)\cdot_\g \huaR(y)=\huaR(\huaR(x)\cdot_\g y+x\cdot_\g \huaR(y))+\lambda
\huaR(x\cdot_\g y),\quad \forall x,y\in \g,\end{equation} then $\huaR$ is called a
{\bf Rota-Baxter operator of weight $\lambda$} on $\g$.
\end{defi}
  Note that a Rota-Baxter operator of weight zero on a
pre-Lie algebra $\g$ is exactly an $\huaO$-operator
associated to the regular representation $(\g;L,R)$.

\begin{pro}\label{pro:RN}
Let $(\g,\cdot_\g)$ be a pre-Lie algebra and $N:\g\longrightarrow
\g$ a linear operator.
\begin{enumerate} \item[\rm(i)] If $N^2=0$, then
$N$ is a Nijenhuis operator if and only if $N$ is a Rota-Baxter
operator of weight zero on $\g$.
\item[\rm(ii)] If $N^2=N$, then $N$ is a
Nijenhuis operator if and only if $N$ is a Rota-Baxter operator of
weight $-1$ on $\g$.
 \item[\rm(iii)] If $N^2={\rm Id}$, then $N$ is a
Nijenhuis operator if and only if $N\pm{\rm Id}$ is a Rota-Baxter
operator of weight $\mp 2$ on $\g$.
\end{enumerate}

\end{pro}

\pf (i) and (ii) are obvious due to the definitions. (iii) follows
from a conclusion in \cite{BGN} (Corollary 2.10). \qed

\begin{ex} {\rm In \cite{LHB}, Rota-Baxter operators $\huaR$ of weight
zero on pre-Lie algebras were studied with some examples. In
particular, there are the following examples satisfying $\huaR^2=0$:
\begin{enumerate}
\item Let $(\g,\cdot_\g)$ be an $n$-dimensional  pre-Lie algebra
with a basis $\{e_1,\cdots, e_n\}$ of $\g$ such that $L_{e_1}={\rm
Id}, L_{e_i}=0,i=2,\cdots,n$. \item Let $(\g,\cdot_\g)$ be an
$n$-dimensional pre-Lie algebra with a basis $\{e_1,\cdots, e_n\}$
of $\g$ such that $R_{e_1}={\rm Id}, R_{e_i}=0,i=2,\cdots,n$.
\item Let $(\g,\cdot_\g)$ be a  $2$-dimensional pre-Lie algebra
with a basis $\{e_1,e_2\}$ whose non-zero products are given by
$$e_2\cdot_\g e_1=-e_1, \quad e_2\cdot_\g e_2=-e_2.$$
 \item Let $(\g,\cdot_\g)$ be a  $2$-dimensional pre-Lie algebra with a basis
$\{e_1,e_2\}$ whose non-zero products are given by
$$e_1\cdot_\g e_2=e_1,\quad  e_2\cdot_\g e_2=e_2.$$
 \item Let $(\g,\cdot_\g)$ be a  $3$-dimensional pre-Lie algebra with a basis
$\{e_1,e_2,e_3\}$ whose non-zero products are given by
$$e_3\cdot_\g e_1=e_1,\quad  e_3\cdot_\g e_2=e_2,\quad e_3\cdot_\g e_3=e_3.$$
\item Let $(\g,\cdot_\g)$ be a  $3$-dimensional pre-Lie algebra
with a basis $\{e_1,e_2,e_3\}$ whose non-zero products are given
by
$$e_1\cdot_\g e_3=-e_1, \quad e_2\cdot_\g e_3=-e_2,\quad e_3\cdot_\g e_3=-e_3.$$
\end{enumerate}
For every above pre-Lie algebra $\g$, $\huaR$ is a Rota-Baxter operator
of weight zero on $\g$ if and only if $\huaR^2=0$. In these cases, such
$\huaR$ is also a Nijenhuis operator on $\g$.
}
\end{ex}

\begin{ex} {\rm
In \cite{AB}, the classification of Rota-Baxter operators $\huaR$
of weight $-1$ on complex associative algebras in dimensions 2 and
3 were given. It is straightforward to find that there are many
examples satisfying $\huaR^2=\huaR$ and hence they are also
Nijenhuis operators on these associative algebras (hence pre-Lie
algebras). }
\end{ex}

\begin{lem}\label{lem:OR} {\rm(\cite{BGN})} Let $\g$ be a pre-Lie algebra and
$(V;\rho,\mu)$ a representation. Let $T:V\rightarrow \g$ be a
linear map. For any $\lambda$, $T$ is an $\huaO$-operator on
$\g$ associated to $(V;\rho,\mu)$ if and only if the linear map
$\huaR_T:=\left(\begin{array}{cc}0&T\\0&-\lambda {\rm
Id}\end{array}\right)$ is a Rota-Baxter operator of weight
$\lambda$ on the semidirect product pre-Lie algebra
$(\g\ltimes_{\rho,\mu}V,\star)$, where the multiplication $\star$
is given by \eqref{eq:mulsemi}.
\end{lem}

\begin{pro}
Let $(\g,\cdot_\g)$ be a pre-Lie algebra and $(V;\rho,\mu)$ a
representation. Let $T:V\rightarrow \g$ be a linear map. Then the
following statements are equivalent.
\begin{enumerate}
\item[\rm(i)] $T$ is an $\huaO$-operator on the pre-Lie algebra $(\g,\cdot_\g)$;
\item[\rm(ii)]
  $N_T :=\left(\begin{array}{cc}0&T\\0&0\end{array}\right)$ is a Nijenhuis operator on
the pre-Lie algerba $(\g\ltimes_{\rho,\mu}V,\star)$;
\item[\rm(iii)]
  $\huaN_T :=\left(\begin{array}{cc}0&T\\0&{\rm Id}\end{array}\right)$ is a Nijenhuis operator on
the pre-Lie algerba $(\g\ltimes_{\rho,\mu}V,\star)$.
\end{enumerate}

\end{pro}

\pf Note that $N_T$ corresponds to $\lambda=0$ and
$(N_T)^2=0$, $\huaN_T$ corresponds to $\lambda=-1$ and
$(\huaN_T )^2=\huaN_T$. Hence the conclusions follow from Proposition \ref{pro:RN} and Lemma \ref{lem:OR}. \qed

\subsection{Compatible $\huaO$-operators and compatible L-dendriform algebras}

\begin{defi}
Let $(\g,\cdot_\g)$ be a pre-Lie algebra and $(V;\rho,\mu)$  a
representation. Let $T_1,T_2: V\longrightarrow \g$ be two
$\huaO$-operators associated to $(V;\rho,\mu)$. If for all
$k_1,k_2$, $k_1T_1+k_2T_2$ is still an $\huaO$-operator
associated to $(V;\rho,\mu)$, then $T_1$ and $T_2$ are called {\bf
compatible}.
\end{defi}

\begin{ex} \label{ex:crs} Let $(\g,\cdot_\g)$ be a pre-Lie algebra. Then two
$\frks$-matrices $r_1$ and $r_2$ are compatible in the sense of
Definition~\ref{defi:cs} if and only $r_1^\sharp$ and $r_2^\sharp$ are
compatible as  $\huaO$-operators on $(\g,\cdot_\g)$
associated to the representation $(\g^*;{\rm ad}^*,-R^*)$.

\end{ex}

\begin{pro}
Let $T_1,T_2: V\longrightarrow \g$ be two $\huaO$-operators
on a pre-Lie algebra $(\g,\cdot_\g)$ associated to a
representation $(V;\rho,\mu)$. Then $T_1$ and $T_2$ are compatible
if and only if the following equation holds:
\begin{eqnarray}
\nonumber T_1(u)\cdot_\g T_2(v)+ T_2(u)\cdot_\g
 T_1(v)&=&T_1\Big(\rho(T_2(u))(v)+\mu(T_2(v))(u)\Big)\\
 &&+T_2\Big(\rho(T_1(u))(v)+\mu(T_1(v))(u)\Big),\quad \forall u,v\in V.\label{eq:CN1}
\end{eqnarray}
\end{pro}

\pf It follows from a direct computation.\qed\vspace{3mm}

Using an $\huaO$-operator and a Nijenhuis operator, we can construct a pair of compatible $\huaO$-operators.

\begin{pro}\label{pro:NT}
Let $T: V\longrightarrow \g$ be an $\huaO$-operator on a
pre-Lie algebra $(\g,\cdot_\g)$ associated to a representation
$(V;\rho,\mu)$ and $N$  a Nijenhuis operator on $(\g,\cdot_\g)$.
Then $N\circ T$ is an $\huaO$-operator on the pre-Lie algebra $(\g,\cdot_\g)$ associated to the representation $(V;\rho,\mu)$
if and only if for all $u,v\in V$, the following equation holds:
\begin{eqnarray}
\nonumber &&N\Big(NT(u)\cdot_\g T(v)+T(u)\cdot_\g NT(v)\Big)\\
&=&N\Big(T\big(\rho(NT(u))(v)+\mu(NT(v))(u)\big)+NT\big(\rho(T(u))(v)+\mu(T(v))(u)\big)\Big).\label{eq:ON}
\end{eqnarray}
In this case, if in addition $N$ is invertible, then $T$ and $N\circ T$
are compatible. More explicitly, for any $\huaO$-operator $T$,
if there exists an invertible Nijenhuis operator $N$ such that $N\circ T$
is also an $\huaO$-operator, then $T$ and $N\circ T$ are compatible.
\end{pro}

\pf   Let $u,v\in V$. Since $N$ is a Nijenhuis operator,
we have
$$NT(u)\cdot_\g NT(v)=N\Big(NT(u)\cdot_\g T(v)+T(u)\cdot_\g
NT(v)\Big)-N^2(T(u)\cdot_\g T(v)).$$ Note that
$$T(u)\cdot_\g T(v)=T\Big(\rho(T(u))(v)+\mu(T(v))(u)\Big).$$
Then
$$NT(u)\cdot_\g NT(v)=NT\Big(\rho(NT(u))(v)+\mu(NT(v))(u)\Big)$$
if and only if  (\ref{eq:ON}) holds.

If $N\circ T$ is an $\huaO$-operator and $N$ is invertible, then we have
$$NT(u)\cdot_\g T(v)+T(u)\cdot_\g NT(v)=
T\big(\rho(NT(u))(v)+\mu(NT(v))(u)\big)+NT\big(\rho(T(u))(v)+\mu(T(v))(u)\big),$$
which is exactly the condition that $N\circ T$ and $T$ are compatible. \qed\vspace{3mm}

A pair of compatible $\huaO$-operators can also give rise to a Nijenhuis operator under some conditions.

\begin{pro}\label{pro:TTN}
Let $T_1,T_2: V\longrightarrow \g$ be two $\huaO$-operators on
a pre-Lie algebra $(\g,\cdot_\g)$ associated to a representation
$(V;\rho,\mu)$. Suppose that $T_2$ is invertible. If $T_1$ and $T_2$
are compatible, then $N=T_1\circ T_2^{-1}$ is a Nijenhuis operator on the pre-Lie algebra $(\g,\cdot_\g)$.
\end{pro}

\pf For all $x,y\in \g$, there exist $u,v\in V$ such that
$T_2(u)=x, T_2(v)=y$. Hence $N=T_1\circ T_2^{-1}$ is a Nijenhuis operator
if and only if the following equation holds:
$$NT_2(u)\cdot_\g NT_2(v)=N(NT_2(u)\cdot_\g T_2(v)+T_2(u)\cdot_\g
NT_2(v))-N^2(T_2(u)\cdot_\g T_2(v)).$$ Since $T_1=N\circ T_2$ is an
$\huaO$-operator, the left hand side of the above equation is
$$NT_2(\rho(NT_2(u))(v)+\mu(NT_2(v))(u)).$$
Since $T_2$ is an $\huaO$-operator which is compatible with
$T_1=N\circ T_2$, we have
\begin{eqnarray*}
 &&NT_2(u)\cdot_\g T_2(v)+T_2(u)\cdot_\g NT_2(v)\\
 &=&
T_2(\rho(NT_2(u))(v)+\mu(NT_2(v))(u))+NT_2(\rho(T_2(u))(v)+\mu(T_2(v))(u))\\
&=&T_2(\rho(NT_2(u))(v)+\mu(NT_2(v))(u))+N(T_2(u)\cdot_\g T_2(v)).
\end{eqnarray*}
Let $N$ act on both sides, we get the conclusion. \qed\vspace{3mm}
%\begin{eqnarray}
%T_1T_2^{-1}(x)\cdot_\g T_1T_2^{-1} (y)&=&
%T_1T_2^{-1}(T_1T_2^{-1}(x)\cdot_\g y+x\cdot_\g
%T_1T_2^{-1}(y))-T_1T_2^{-1}T_1T_2^{-1}(x\cdot_\g y).
%\end{eqnarray}
%Let $T_2T_1^{-1}$ multiplies both sides of the above equation. Then
%the above equation is exactly the following equation:
%\begin{eqnarray}
%T_2T_1^{-1}(T_1(u)\cdot_\g T_1(v)) =T_1(u)\cdot_\g
%T_2(v)+T_2(u)\cdot_g T_1(v)-T_1T_2^{-1}(T_2(u)\cdot_\g
%T_2(v)).\label{eq:CN2}
%\end{eqnarray}
%Since both $T_1$ and $T_2$ are $\huaO$-operators,
%Eq.~(\ref{eq:CN2}) is exactly Eq.~(\ref{eq:CN1}). Hence the
%conclusion holds. \qed

By Proposition \ref{pro:NT} and \ref{pro:TTN}, we have
\begin{cor}
Let $T_1,T_2: V\longrightarrow \g$ be two $\huaO$-operators on
a pre-Lie algebra $(\g,\cdot_\g)$ associated to a representation
$(V;\rho,\mu)$. Suppose that  $T_1$ and $T_2$ are invertible. Then
$T_1$ and $T_2$ are compatible if and only if $N=T_1\circ T_2^{-1}$ is a
Nijenhuis operator.
\end{cor}

Since a Rota-Baxter operator of weight zero is an $\mathcal
O$-operator on a pre-Lie algebra associated to the regular
representation, we have the following conclusion.

\begin{cor} Let $(\g,\cdot_\g)$ be a pre-Lie algebra. Suppose that
$\huaR_1$ and $\huaR_2$ are two invertible Rota-Baxter operators of weight
zero. Then $\huaR_1$ and $\huaR_2$ are compatible in the sense that any
linear combination of $\huaR_1$ and $\huaR_2$ is still a Rota-Baxter
operator of weight zero if and only if $N=\huaR_1\circ \huaR_2^{-1}$ is a
Nijenhuis operator.
\end{cor}

\begin{rmk}
Note that the inverse of an invertible Rota-Baxter operator is a
derivation on a pre-Lie algebra. However, it is interesting to see
that the similar role of Nijenhuis operators  as above is not
available for   derivations since any linear combination of two
derivations is always a derivation due to the fact that the set of
all derivations is a linear space.
\end{rmk}

  There is a notion of L-dendriform algebra which was
introduced in \cite{BaiO} as the algebraic structure behind an
$\huaO$-operators on a pre-Lie algebra. At the end of this
section, we introduce the notion of compatible L-dendriform
algebras as the naturally induced algebraic structures from the
compatible $\huaO$-operators on a pre-Lie algebra and hence in
particular, there are compatible L-dendriform algebras obtained
from a pair of an $\huaO$-operator and  a
Nijenhuis operator satisfying some conditions.

\begin{defi}{ Let $A$ be a vector space with two
binary operations denoted by $\triangleright$ and $ \triangleleft:
A\otimes A \rightarrow A$.  $(A, \triangleright, \triangleleft)$ is
called an {\bf L-dendriform algebra} if for all $x,y,z\in A$,
\begin{eqnarray}
x\triangleright(y\triangleright z)&=&(x\triangleright y)\triangleright
z + (x\triangleleft y)\triangleright z +
y\triangleright(x\triangleright z) - (y\triangleleft
x)\triangleright z - (y\triangleright x)\triangleright z,\\
x\triangleright(y\triangleleft z)&=&(x\triangleright y)\triangleleft z
+ y\triangleleft (x\triangleright z) + y\triangleleft(x\triangleleft
z) - (y\triangleleft x)\triangleleft z.
\end{eqnarray}
Two L-dendriform algebras $(A, \triangleright_1, \triangleleft_1)$
and $(A, \triangleright_2, \triangleleft_2)$ are called {\bf
compatible} if for all $k_1,k_2$, $(A,
k_1\triangleright_1+k_2\triangleright_2,
k_1\triangleleft_1+k_2\triangleleft_2)$ is an L-dendriform algebra.
We denote it by $(A, \triangleright_1, \triangleleft_1;
\triangleright_2, \triangleleft_2)$. }\end{defi}

For an L-dendriform algebra $(A, \triangleright, \triangleleft)$,
the binary operation $\cdot_A:A\otimes A\rightarrow A$ given by
\begin{equation}
x\cdot_A y=x\triangleright y - y \triangleleft x,\quad \forall x, y \in A,
\end{equation}
defines a (vertical)  pre-Lie algebra.

\begin{lem}
Two L-dendriform algebras $(A, \triangleright_1, \triangleleft_1)$
and $(A, \triangleright_2, \triangleleft_2)$ are compatible if and
only if for all $x,y,z\in A$, the following equations hold:
\begin{eqnarray*}
\nonumber x\triangleright_1(y\triangleright_2 z)+x\triangleright_2(y\triangleright_1 z)&=&(x\triangleright_1 y)\triangleright_2+(x\triangleright_2 y)\triangleright_1 z + (x\triangleleft_1 y)\triangleright_2 z
+(x\triangleleft_2 y)\triangleright_1 z\\
\nonumber&&+y\triangleright_1(x\triangleright_2 z)
+y\triangleright_2(x\triangleright_1 z) -(y\triangleleft_1
x)\triangleright_2 z- (y\triangleleft_2
x)\triangleright_1 z \\
&&- (y\triangleright_1 x)\triangleright_2 z-(y\triangleright_2 x)\triangleright_1 z;\\
\nonumber x\triangleright_1(y\triangleleft_2 z)+x\triangleright_2(y\triangleleft_1 z)&=&(x\triangleright_1 y)\triangleleft_2 z+(x\triangleright_2 y)\triangleleft_1 z+ y\triangleleft_1 (x\triangleright_2 z) + y\triangleleft_2 (x\triangleright_1 z)\\
&&+ y\triangleleft_2(x\triangleleft_1
z) + y\triangleleft_1(x\triangleleft_2
z)- (y\triangleleft_1 x)\triangleleft_2 z- (y\triangleleft_2 x)\triangleleft_1 z.
\end{eqnarray*}
\end{lem}
\pf It is straightforward. \qed

\begin{lem}{\rm(\cite{BaiO})}\label{le:OL}
Let $T: V\longrightarrow \g$ be an $\huaO$-operator on a
pre-Lie algebra $(\g,\cdot_\g)$ associated to a representation
$(V;\rho,\mu)$. Then there exists an L-dendriform algebra structure
on $V$ defined by
\begin{equation}
u \triangleright v = \rho(T(u))v,\quad u \triangleleft v = -\mu(T(u))v,\quad
\forall u, v\in V.
\end{equation}  There is an induced L-dendriform algebra structure
on $T(V) = \{T(v)\mid v\in V\}\subset \g$ given by
\begin{equation}
T(u) \triangleright T(v) = T(u \triangleright v),\quad T(u) \triangleleft
T(v) = T(u \triangleleft v), \quad \forall u, v \in V.
\end{equation}
If in addition, $T$ is invertible, then there exists an L-dendriform
algebra structure on $\g$ defined by
\begin{equation}
x \triangleright y = T\rho(x)T^{-1}(y),\quad x \triangleleft y =
-T\mu(x)T^{-1}(y), \quad\forall x, y\in \g.
\end{equation}
\end{lem}

\begin{pro}\label{pro:ooLd}
Let $T_1,T_2: V\longrightarrow \g$ be two $\huaO$-operators on a
pre-Lie algebra $(\g,\cdot_\g)$ associated to a representation
$(V;\rho,\mu)$. Assume that $T_1$ and $T_2$ are compatible. Then
$(V, \triangleright_1, \triangleleft_1)$ and $(V,
\triangleright_2, \triangleleft_2)$ are compatible L-dendriform
algebras, where
\begin{eqnarray*}
 u \triangleright_1 v &=& \rho(T_1(u))v, \quad u \triangleleft_1 v =
-\mu(T_1(u))v, \\
u \triangleright_2 v &=& \rho(T_2(u))v,\quad u
\triangleleft_2 v = -\mu(T_2(u))v, \quad \forall u, v \in V.
\end{eqnarray*}
If in addition, both $T_1$ and $T_2$ are invertible, then $(\g,
\triangleright_1, \triangleleft_1)$ and $(\g, \triangleright_2,
\triangleleft_2)$ are compatible L-dendriform algebras, where
\begin{eqnarray*}
x \triangleright_1 y &=& T_1\rho(x)T_1^{-1}(y),\quad x \triangleleft_1 y =
-T_1\mu(x)T_1^{-1}(y), \\
x \triangleright_2 y &=&T_2\rho(x)T_2^{-1}(y),\quad x \triangleleft_2 y =
-T_2\mu(x)T_2^{-1}(y),\quad \forall x, y\in \g.
 \end{eqnarray*}
\end{pro}

\pf It follows directly from Lemma~\ref{le:OL} and the definitions of compatible
$\huaO$-operators and compatible L-dendriform algebras. \qed

\begin{cor}\label{co:BN}
Let $T: V\longrightarrow \g$ be an $\huaO$-operator on a
pre-Lie algebra $(\g,\cdot_\g)$ associated to a representation
$(V;\rho,\mu)$. If there exists an invertible Nijenhuis operator $N$
such that $N\circ T$ is also an $\huaO$-operator on $(\g,\cdot_\g)$
associated to $(V;\rho,\mu)$, then $(V, \triangleright_1,
\triangleleft_1)$ and $(V, \triangleright_2, \triangleleft_2)$ are
compatible L-dendriform algebras, where
\begin{eqnarray*}
u \triangleright_1 v &=& \rho(T(u))v,\quad u \triangleleft_1 v = -\mu(T(u))v,\\
 u \triangleright_2 v &=& \rho(NT(u))v,\quad u \triangleleft_2 v = -\mu(NT(u))v, \quad \forall u, v\in V.
\end{eqnarray*}
If in addition, $T$ is invertible, then  $(\g, \triangleright_1,
\triangleleft_1)$ and $(\g, \triangleright_2, \triangleleft_2)$ are
compatible L-dendriform algebras, where
\begin{eqnarray*}
x \triangleright_1 y &=& T\rho(x)T^{-1}(y),\quad x \triangleleft_1 y =
-T\mu(x)T^{-1}(y), \\
x \triangleright_2 y &=& NT\rho(x)(NT)^{-1}(y),\quad
x\triangleleft_2 y = -NT\mu(x)(NT)^{-1}(y),\quad \forall x, y\in \g.
\end{eqnarray*}
\end{cor}

\begin{ex}{\rm
Let $\huaR$ be a Rota-Baxter operator of weight zero on a pre-Lie
algebra $(\g,\cdot_\g)$. Suppose that there exists an invertible
Nijenhuis operator $N$ such that  $N\circ \huaR$ is also a Rota-Baxter
operator of weight zero, then $(\g, \triangleright_1,
\triangleleft_1)$ and $(\g, \triangleright_2, \triangleleft_2)$
are compatible L-dendriform algebras, where
\begin{eqnarray*}
x \triangleright_1 y &=& \huaR(x)\cdot_\g y,\quad x \triangleleft_1 y = -y\cdot_\g \huaR(x),\\
 x \triangleright_2 y &=& (N\circ \huaR)(x) \cdot_\g y,\quad x \triangleleft_2 y = -y\cdot_\g (N\circ \huaR)(x), \quad \forall x, y\in \g.
\end{eqnarray*}
If in addition, $\huaR$ is invertible, then  $(\g, \triangleright_1,
\triangleleft_1)$ and $(\g, \triangleright_2, \triangleleft_2)$
are also compatible L-dendriform algebras, where
\begin{eqnarray*}
x \triangleright_1 y &=& \huaR(x\cdot_\g \huaR^{-1}(y)),\quad x
\triangleleft_1 y =
-\huaR(\huaR^{-1}(y)\cdot_\g x), \\
x \triangleright_2 y &=& (N\circ \huaR)(x\cdot_\g (N\circ
\huaR)^{-1}(y)),\quad x\triangleleft_2 y = -(N\circ \huaR)((N\circ
\huaR)^{-1}(y)\cdot_\g x),\quad \forall x, y\in \g.
\end{eqnarray*}
}
\end{ex}

\section{Pseudo-Hessian-Nijenhuis pre-Lie
algebras}   In this section, we consider the case of the dual
representation of the regular representation of a pre-Lie algebra
for the study in the previous section, that is, the compatible
$\frks$-matrices (see Example~\ref{ex:crs}) and Nijenhuis
operators. In particular, there is an interesting geometric
interpretation here. Note that the inverse of an $\frks$-matrix
corresponds to a pseudo-Hessian structure. So it is natural to
introduce the notion of a pseudo-Hessian-Nijenhuis structure on a
pre-Lie algebra as a pair of a pseudo-Hessian structure and a
Nijenhuis operator satisfying  certain compatibility conditions.
Such a structure  induces a sequence of pseudo-Hessian structures
as well as pseudo-Hessian-Nijenhuis structures. We show that two
compatible invertible $\frks$-matrices   give rise to a
pseudo-Hessian-Nijenhuis structure. Conversely,  a
pseudo-Hessian-Nijenhuis structure must be obtained in this way.
Finally, we investigate the pre-Lie algebra structure on $\g^*$
associated to a pair of compatible $\frks$-matrices  on a pre-Lie algebra $(\g,\cdot_\g)$.

Let $(\g,\cdot_\g)$ be a pre-Lie algebra and $N:\g\longrightarrow \g$ a   linear operator. Suppose that $\frkB\in \Sym^2(\g^*)$ satisfies that
\begin{equation}\label{eq:Hess1}
\frkB(N(x),y)=\frkB(x,N(y)),\quad \forall~x,y\in\g.
\end{equation}
Define a sequence of bilinear forms by $\frkB_k(x,y)=\frkB(x,N^k(y))$ for all $k\in \Integ.$
It is obvious that $\frkB_k\in \Sym^2(\g^*)$.

\begin{defi}
Let $\frkB$ be a pseudo-Hessian structure and $N$ an invertible Nijenhuis operator on a pre-Lie algebra $(\g,\cdot_\g)$. Then $(\frkB,N)$ is called a {\bf pseudo-Hessian-Nijenhuis structure} on the pre-Lie algebra  $(\g,\cdot_\g)$ if  \eqref{eq:Hess1} holds and $ \dM^T\frkB_1=0$. More precisely,
$$
\frkB(x\cdot_\g y,N(z))-\frkB(x, N(y\cdot_\g z))=\frkB(y\cdot_\g
x,N(z))-\frkB(y,N(x\cdot_\g z)),\quad \forall~x,y,z\in\g.
$$

  A pre-Lie algebra $(\g,\cdot_\g)$ equipped with a pseudo-Hessian-Nijenhuis structure $(\frkB,N)$ is called a {\bf pseudo-Hessian-Nijenhuis pre-Lie algebra}.
\end{defi}

\begin{rmk}   Later we will see that many properties for a
pseudo-Hessian-Nijenhuis pre-Lie algebra are totally parallel to
the ones for a symplectic-Nijenhuis Lie algebra.
The notion of a symplectic-Nijenhuis structure on a Lie algebra,
or more generally on a Lie algebroid, was introduced in
\cite[Definition 8.1]{Kos} as generalizations of Poisson-Nijenhuis
structures \cite{KosPN,MM}. Let $(\g,[\cdot,\cdot]_\g,\omega)$ be
a symplectic Lie algebra and $N$ a Nijenhuis operator on the Lie
algebra $(\g,[\cdot,\cdot]_\g)$. A pair $(\omega,N)$ is called a
symplectic-Nijenhuis structure\footnote{It is called an $\Omega
N$-structure in \cite{Kos}.} if
\begin{eqnarray}\label{eq:sn}
\omega(N(x),y)=\omega(x,N(y)),\quad \forall~x,y\in\g
\end{eqnarray}
and $\omega_N$ is a $2$-cocycle on the Lie algebra $(\g,[\cdot,\cdot]_\g)$, where $\omega_N:\otimes^2 \g\longrightarrow \g $ is defined by $\omega_N(x,y)=\omega(N(x),y)$. By \eqref{eq:sn}, $\omega_N$ is skew-symmetric.
%\begin{eqnarray}
 %\{\omega, \omega_N\}_r=0,\quad \{\omega_N,\omega_N\}_r=0.
%\end{eqnarray}
%where $\omega_N(x,y)=\omega(Nx,y)$, $\{\cdot,\cdot\}_r$ is the Schouten bracket associated to the Lie algebra structure on $\g^*$ given by an $r$-matrix $r$, in which $(r^\sharp)^{-1}:\g^*\longrightarrow\g$ is defined by $\langle(r^\sharp)^{-1}(\xi),\eta\rangle=\omega(\xi,\eta)$.
\end{rmk}

Now we give some examples for pseudo-Hessian-Nijenhuis pre-Lie algebras.
\begin{ex}{\rm
   Let
$(\g,\cdot_\g)$ be a 2-dimensional pre-Lie algebra with a basis
$\{e_1,e_2\}$ whose non-zero products are given as follows:
$$e_2\cdot e_1=-e_1,\quad e_2\cdot e_2=e_2.$$
  Let $\{e_1^*,e_2^*\}$ be the dual basis and $\frkB$   a
nondegenerate symmetric bilinear form given by
  $$\frkB=a e_1^*\otimes e_2^*+ a e_2^*\otimes e_1^*+b e_2^*\otimes e_2^*,\quad a\neq 0.$$
  Then $(\g,\cdot,\frkB)$ is a pseudo-Hessian pre-Lie algebra.

It is straightforward to verify that all $ N =\begin{bmatrix}c&
d\\ 0&c\end{bmatrix} $ with $c\neq0$ are  invertible Nijenhuis
  operators  on the pre-Lie algebra $(\g,\cdot)$.

By direct calculation, we have
\begin{eqnarray*}
\frkB(N(x),y)&=&\frkB(x,N(y)),\\
 \dM^T\frkB_1(x,y,z)&=&0,\quad \forall~x,y,z\in\g,
\end{eqnarray*}
where $\frkB_1(x,y)=\frkB(x,N(y))$. Thus $(\frkB,N)$ is a pseudo-Hessian-Nijenhuis structure on the pre-Lie algebra $(\g,\cdot)$.
}
\end{ex}

\begin{ex}{\rm
    Let $(\g,\cdot_\g)$ be a 3-dimensional pre-Lie algebra
with a basis $\{e_1,e_2,e_3\}$ whose non-zero products are given
as follows:
$$e_3\cdot e_2=e_2,\quad e_3\cdot e_3=-e_3.$$
  Let $\{e_1^*,e_2^*,e_3^*\}$ be the dual basis and $\frkB$
  a nondegenerate symmetric bilinear form given by
  $$\frkB=a e_1^*\otimes e_1^*+ b e_2^*\otimes e_3^*+b e_3^*\otimes e_2^*+ce_3^*\otimes e_3^*,\quad ab\neq 0.$$
  Then $(\g,\cdot,\frkB)$ is a pseudo-Hessian pre-Lie algebra.

It is straightforward to verify that all $ N =\begin{bmatrix}d&
0&0\\ 0&e&f\\0&0&e\end{bmatrix} $ with $de\neq0$ are
invertible Nijenhuis  operators  on the pre-Lie
algebra $(\g,\cdot)$.

By direct calculation, we have
\begin{eqnarray*}
\frkB(N(x),y)&=&\frkB(x,N(y)),\\
 \dM^T\frkB_1(x,y,z)&=&0,\quad \forall~x,y,z\in\g,
\end{eqnarray*}
where $\frkB_1(x,y)=\frkB(x,N(y))$. Thus $(\frkB,N)$ is a pseudo-Hessian-Nijenhuis structure on the pre-Lie algebra $(\g,\cdot)$.
}
\end{ex}

We give a useful lemma.

\begin{lem}
With above notations, for $\frkB_k\in \Sym^2(\g^*),~k\in\Integ$, we have
\begin{eqnarray}
&&\dM^T\frkB_{k+1}(x,y,z)=\dM^T\frkB_{k}(N(x),y,z)+\dM^T\frkB_{k}(x,N(y),z)-\dM^T\frkB_{k-1}(N(x),N(y),z)\nonumber\\
&&\qquad-\frkB\Big([N(x),N(y)]^c-N([N(x),y]^c+[x,N(y)]^c-N([x,y]^c),N^{k-1}(z)\Big).
\label{eq:Hess2}
\end{eqnarray}
Here we set $\frkB_0=\frkB$.
\end{lem}
\pf By the definition of the coboundary operator $\dM^T$ associated to the  trivial representation, we have
\begin{eqnarray*}
\dM^T\frkB_{k+1}(x,y,z)&=&-\frkB(y,N^{k+1}(x\cdot_\g z))+\frkB(x,N^{k+1}(y\cdot_\g z))-\frkB([x,y]^c,N^{k+1}(z)),\\
\dM^T\frkB_{k}(Nx,y,z)&=&-\frkB(y,N^{k}(N(x)\cdot_\g z))+\frkB(N(x),N^{k}(y\cdot_\g z))-\frkB([N(x),y]^c,N^{k}(z)),\\
\dM^T\frkB_{k}(x,Ny,z)&=&-\frkB(N(y),N^{k}(x\cdot_\g z))+\frkB(x,N^{k}(N(y)\cdot_\g z))-\frkB([x,N(y)]^c,N^{k}(z)),\\
\dM^T\frkB_{k-1}(Nx,Ny,z)&=&-\frkB(N(y),N^{k-1}(N(x)\cdot_\g z))+\frkB(N(x),N^{k-1}(N(y)\cdot_\g z))\\
&&-\frkB([N(x),N(y)]^c,N^{k-1}(z)).
\end{eqnarray*}
Substituting the expression obtained into the right hand side of \eqref{eq:Hess2} and take \eqref{eq:Hess1} into account, we get the desired equality.\qed

\begin{thm}\label{thm:Hess1}
Let $(\frkB,N)$ be a pseudo-Hessian-Nijenhuis structure on a pre-Lie algebra $(\g,\cdot_\g)$. Then for all $k\in\Integ$, $\frkB_k$ is a pseudo-Hessian structure on $(\g,\cdot_\g)$, i.e. $\dM^T\frkB_k=0$. Furthermore, for all $k,l\in\Integ$, $(\frkB_k,N^l)$  is a pseudo-Hessian-Nijenhuis structure on $(\g,\cdot_\g)$.
\end{thm}
\begin{proof}
  Since $N$ is a Nijenhuis operator on the pre-Lie algebra $(\g,\cdot_\g)$, $N$ is also the Nijenhuis operator for the corresponding sub-adjacent Lie algebra $(\g,[\cdot,\cdot]^c)$. Thus by \eqref{eq:Hess2}, we have
  \begin{eqnarray*}
\dM^T\frkB_{k+1}(x,y,z)&=&\dM^T\frkB_{k}(N(x),y,z)+\dM^T\frkB_{k}(x,N(y),z)-\dM^T\frkB_{k-1}(N(x),N(y),z).
\end{eqnarray*}
Because of $\dM^T\frkB=\dM^T\frkB_1=0$, it follows that $\dM^T\frkB_2=0$. By induction on $k$, it follows that $\dM^T\frkB_k=0$ for all $k\in\Integ$. We finish the proof.
\end{proof}

  In the following, we characterize the algebraic structure
of a pseudo-Hessian-Nijenhuis structure on a pre-Lie algebra in
terms of compatible structures. At first, two invertible
compatible $\frks$-matrices can induce a pseudo-Hessian-Nijenhuis
structure on a pre-Lie algebra.

\begin{pro}\label{pro:smatrixphn}
Let $(\g,\cdot_\g)$ be a pre-Lie algebra and $r_1,r_2$ two compatible $\frks$-matrices. Suppose that $r_2$ is invertible.
Then $N=r_1^\sharp\circ (r_2^\sharp)^{-1}:\g\longrightarrow\g $ is a Nijenhuis operator on the pre-Lie algebra $(\g,\cdot_\g)$. Furthermore, if  $r_1$ is also invertible,
then $(\g,\cdot_\g,\frkB,N)$ is a pseudo-Hessian-Nijenhuis pre-Lie algebra, where $\frkB$ is defined by \eqref{eq:rB} associated to the $\frks$-matrix $r_1$.
\end{pro}
\pf Since $r_1$ and $r_2$ are $\frks$-matrices, $r_1^\sharp$ and
$r_2^\sharp$ are $\huaO$-operators  associated to
the representation $(\g^*;{\rm ad}^*, -R^*)$. Then the first
conclusion follows from Proposition \ref{pro:TTN}. As for the
second conclusion, by Proposition \ref{pro:Hessian1}, if $r_1$ is
invertible, then $\frkB$ is a pseudo-Hessian structure. As both
$r_1$ and $r_2$ are symmetric, we have
$$\frkB(N(x),y)=\frkB(x,N(y)).$$
 Since $r_2$ is  an invertible $\frks$-matrix, $\frkB_1\in\Sym^2(\g^*)$ defined by $\frkB_1(x,y)=\frkB(x,N(y))=\langle x,(r_2^\sharp)^{-1}y\rangle$ is closed and therefore $(\g,\cdot_\g,\frkB,N)$ is a pseudo-Hessian-Nijenhuis pre-Lie algebra. \qed
\vspace{3mm}

Conversely, given a pseudo-Hessian-Nijenhuis structure on a pre-Lie algebra, we can find a pair of compatible $\frks$-matrices.
\begin{pro}\label{thm:pseudoHS}
Let $(\g,\cdot_\g,\frkB,N)$ be a pseudo-Hessian-Nijenhuis pre-Lie
algebra. Then there exists a pair of compatible $\frks$-matrices
$r_1$ and $r_2$ producing $(\frkB,N)$ as described in Proposition
\ref{pro:smatrixphn}.
\end{pro}
\pf Since $\frkB$ is nondegenerate,  the operator $\frkB^\natural:\g\longrightarrow\g^*$ defined by
$$\frkB^\natural (x)=\iota_x\frkB,\quad \forall x\in\g$$
is an isomorphism. Putting $r_1^\sharp=(\frkB^\natural)^{-1},\  r_2^\sharp=N^{-1}\circ (\frkB^\natural)^{-1}$. Since $\frkB$ and $\frkB_1$ are closed, $r_1$ and $r_2$ are $\frks$-matrices. Since $N$ is an invertible Nijenhuis operator, by Proposition \ref{pro:NT}, $r_1$ and $r_2$ are compatible.\qed\vspace{3mm}

 By Proposition  \ref{thm:pseudoHS} and Corollary
\ref{co:BN},
  a pseudo-Hessian-Nijenhuis pre-Lie
algebra gives rise to a pair of compatible L-dendriform algebras.

\begin{cor}
Let $(\g,\cdot_\g,\frkB,N)$ be a pseudo-Hessian-Nijenhuis pre-Lie
algebra. Then there exists a pair of compatible L-dendriform algebra
structures on $\g$ given by
\begin{equation}
\frak {B}(x\triangleright_1 y,z)= -\frak {B}(y,x\cdot_\g z-z\cdot_\g
x),\quad  \frak  {B}(x\triangleleft_1 y,z)= -\frak {B}(y,z\cdot_\g x),
\end{equation}
\begin{equation}\label{eq:Ldendriform2}
\frak {B}(x\triangleright_2 y,z)= -\frak {B}(N(y),x\cdot_\g
N^{-1}(z)-N^{-1}(z)\cdot_\g x),\quad  \frak  {B}(x\triangleleft_2 y,z)= -\frak
{B}(N(y),N^{-1}(z)\cdot_\g x).
\end{equation}

%such that $(\g,\cdot_\g)$ is the associated vertical pre-Lie
%algebra.
\end{cor}
\pf Since $(\g,\cdot_\g,\frkB,N)$ is a pseudo-Hessian-Nijenhuis
pre-Lie algebra, by Proposition \ref{thm:pseudoHS},
$r_1^\sharp=(\frkB^\natural)^{-1}$ and $r_2^\sharp=N^{-1} \circ
(\frkB^\natural)^{-1}$ are invertible $\frks$-matrices and
compatible. Thus $r_1^\sharp$ and $r_2^\sharp$ are also compatible
$\huaO$-operators on the pre-Lie algebra $(\g,\cdot_\g)$
associated to the representation $(\ad^*,-R^*)$.  By  Corollary \ref{co:BN}, there is a pair of
compatible L-dendriform algebras given by
  \begin{eqnarray*}
x \triangleright_1 y &=& r_1^\sharp(\ad^*_x(r_1^\sharp)^{-1}(y)),\quad x \triangleleft_1 y =
r_1^\sharp(R^*_x(r_1^\sharp)^{-1}(y)), \\
x \triangleright_2 y &=& r_2^\sharp(\ad^*_x(r_2^\sharp)^{-1}(y)),\quad x \triangleleft_2 y =
r_2^\sharp(R^*_x(r_2^\sharp)^{-1}(y)),\quad \forall x, y\in \g.
 \end{eqnarray*}
 We only verify  \eqref{eq:Ldendriform2}. The other one can be proved similarly. In fact, it follows from
 \begin{eqnarray*}
 \frkB(x \triangleright_2 y,z)&=&\langle (r_1^\sharp)^{-1}( r_2^\sharp(\ad^*_x(r_2^\sharp)^{-1}(y)),z\rangle=\langle \ad^*_x(r_2^\sharp)^{-1}(y), r_2^\sharp((r_1^\sharp)^{-1}(z))\rangle\\
 &=&\langle \ad^*_x(r_2^\sharp)^{-1}(y),N (z)\rangle=-\langle (r_2^\sharp)^{-1}(y),[x,N^{-1}(z)]^c\rangle\\
 &=&-\langle (r_1^\sharp)^{-1}(r_1^\sharp ((r_2^\sharp)^{-1}(y))),[x,N^{-1} (z)]^c\rangle=-\langle(r_1^\sharp)^{-1}(N(y)),[x,N^{-1}(z)]^c\rangle\\
 &=&-\frak {B}(N(y),x\cdot_\g
N^{-1}(z)-N^{-1} (z)\cdot_\g x),
 \end{eqnarray*}
and
\begin{eqnarray*}
\frak{B}(x\triangleleft_2 y,z)&=&\langle(r_1^\sharp)^{-1}( r_2^\sharp(R^*_x(r_2^\sharp)^{-1}(y)),z\rangle=\langle R^*_x(r_2^\sharp)^{-1}(y), r_2^\sharp((r_1^\sharp)^{-1}(z))\rangle\\
&=&\langle R^*_x(r_2^\sharp)^{-1}(y), N^{-1}(z)\rangle=-\langle(r_2^\sharp)^{-1}(y),N^{-1}(z)\cdot_\g x \rangle\\
&=&-\langle(r_1^\sharp)^{-1}(r_1^\sharp ((r_2^\sharp)^{-1}(y))),N^{-1}(z)\cdot_\g x \rangle=-\langle(r_1^\sharp)^{-1}(N(y)),N^{-1}(z)\cdot_\g x\rangle\\
&=&-\frak{B}(N(y),N^{-1}(z)\cdot_\g x).
 \end{eqnarray*}
 We finish the proof.
\qed\vspace{2mm}

By Theorem \ref{thm:Hess1}, we already know that a
pseudo-Hessian-Nijenhuis structure $(\g,\cdot_\g,\frkB,N)$
produces a sequence of pseudo-Hessian-Nijenhuis structures $(
\frkB,N^n).$ By Proposition \ref{pro:smatrixphn}, any compatible
invertible $\frks$-matrices $r_1$ and $r_2$ produce a sequence of
$\frks$-matrices $s_n\in\Sym^2(\g)$ satisfying
$$(s_n)^\sharp=(r_2^\sharp\circ(r_1^\sharp)^{-1})^n\circ r_1^\sharp.$$
In fact, since $N=r_1^\sharp \circ (r_2^\sharp)^{-1}$, we have
\begin{eqnarray*}
\frkB(x,N^n(y))&=&\frkB(x,(r_1^\sharp \circ
(r_2^\sharp)^{-1})^n(y))=\langle x,(r_1^\sharp)^{-1}\circ
(r_1^\sharp \circ (r_2^\sharp)^{-1})^n(y)\rangle.
\end{eqnarray*}
Since $(\g,\cdot_\g,\frkB,N^n)$ is a pseudo-Hessian-Nijenhuis
pre-Lie algebra, $\frkB_n\in\Sym^2(\g^*)$ defined by
$\frkB_n(x,y)=\frkB(x,N^n(y))$ is closed. Thus by Proposition
\ref{pro:Hessian1}, the inverse of
$(r_1^\sharp)^{-1}\circ(r_1^\sharp\circ (r_2^\sharp)^{-1})^n$ is
an $\frks$-matrix, i.e. $s_n$ is an $\frks$-matrix. In particular,
we have $s_0=r_1$ and $s_1=r_2.$

\begin{rmk}
  All $\frks$-matrices $s_n$ given   above are compatible with $r_1$ and also with each other. In fact, by the fact that $(\g,\cdot_\g,\frkB,N^n)$ is a pseudo-Hessian-Nijenhuis structure and the proof of Proposition \ref{thm:pseudoHS},   $(\frkB^\natural)^{-1}$ and $(N^{-1})^n\circ(\frkB^\natural)^{-1}$ are compatible, which means that $r_1$ and $s_n$ are compatible. Similarly, by the fact that $(\g,\cdot_\g,\frkB_n,N^m)$ is a pseudo-Hessian-Nijenhuis pre-Lie algebra, $s_n$ and  $s_m$ are also compatible.
\end{rmk}

At the end of this section, we consider the pre-Lie algebra
structure on the dual space $\g^*$ associated to a pair of
compatible $\frks$-matrices on a pre-Lie algebra $(\g,\cdot_\g)$.

\begin{pro}\label{pro:dual}
Let $(\g,\cdot_\g)$ be a pre-Lie algebra and $r_1,r_2$ two compatible $\frks$-matrices in which $r_2$ is invertible. Then $N^*=(r^\sharp_2)^{-1} \circ r_1^\sharp$ is a Nijenhuis operator on the pre-Lie algebra $(\g^*,\cdot_{r_2})$, where the multiplication $\cdot_{r_2}$ is given by \eqref{eq:pre-bia}.
 Furthermore, the corresponding trivial deformation of $(\g^*,\cdot_{r_2})$ generated by $N^*$ is given by
\begin{equation}\label{eq:pre-bia2}
\xi\cdot_t \eta=\xi\cdot_{r_2} \eta+t\xi\cdot_{r_1}\eta,\quad \forall~\xi,\eta\in\g^*.
\end{equation}
 \end{pro}
\pf  Since $r_1$ and $r_2$ are compatible, $r_1+tr_2$ is an $\frks$-matrix for arbitrary $t$. By Proposition \ref{pro:morphism}, we have
$$(r_1^\sharp+tr_2^\sharp)(\xi\cdot_{r_1+tr_2}\eta)=(r_1^\sharp+tr_2^\sharp)(\xi)\cdot_\g (r_1^\sharp+tr_2^\sharp)(\eta),$$
which means that
\begin{eqnarray*}
r_1^\sharp(\xi\cdot_{r_1}\eta)&=&r_1^\sharp(\xi)\cdot_\g r_1^\sharp(\eta),\\
r_2^\sharp(\xi\cdot_{r_2}\eta)&=&r_2^\sharp(\xi)\cdot_\g r_2^\sharp(\eta),\\
r_2^\sharp(\xi\cdot_{r_1}\eta)+r_1^\sharp(\xi\cdot_{r_2}\eta)&=&r_1^\sharp(\xi)\cdot_\g r_2^\sharp(\eta)+r_2^\sharp(\xi)\cdot_\g r_1^\sharp(\eta).
\end{eqnarray*}
Now applying $(r_2^\sharp)^{-1}$ to both sides of above equalities, we obtain from the first two equalities
\begin{equation}\label{eq:pre-bia3}
N^*(\xi\cdot_{r_1}\eta)=N^*(\xi)\cdot_{r_2}N^*(\eta).
\end{equation}
By the second and the third equalities, we get
\begin{eqnarray}
\xi\cdot_{r_1}\eta+N^*(\xi\cdot_{r_2}\eta)&=&(r_2^\sharp)^{-1}(r_1^\sharp(\xi)\cdot_\g (r_2^\sharp)(\eta))+(r_2^\sharp)^{-1}(r^\sharp_2(\xi)\cdot_\g r_1^\sharp(\eta))\nonumber\\
&=&(r_2^\sharp)^{-1} r_1^\sharp(\xi)\cdot_{ r_2}\eta+\xi\cdot_{ r_2} (r_2^\sharp)^{-1} r_1^\sharp(\eta)\nonumber\\
&=&N^*(\xi)\cdot_{ r_2}\eta+\xi\cdot_{ r_2} N^*(\eta)\label{eq:pre-bia4}.
\end{eqnarray}
By \eqref{eq:pre-bia3} and \eqref{eq:pre-bia4}, we obtain that $N^*$ is a Nijenhuis operator on the pre-Lie algebra $(\g^*,\cdot_{r_2})$.

It is straightforward to see that the pre-Lie algebra structure given by \eqref{eq:pre-bia} associated to the $\frks$-matrix $r_2+tr_1$ is exactly the one given by \eqref{eq:pre-bia2}. Thus, \eqref{eq:pre-bia2} is a deformation of the pre-Lie algebra $(\g^*,\cdot_{r_2})$. By \eqref{eq:pre-bia4}, this deformation is generated by $N^*$. The proof is finished.\qed

\begin{cor}
Let $(\g,\cdot_\g,\frkB,N)$ be a pseudo-Hessian-Nijenhuis pre-Lie
algebra. Then the conjugate $N^*:\g^*\longrightarrow \g^*$ is a Nijenhuis operator on the pre-Lie
algebra $(\g^*,\cdot_{\g^*})$, where the multiplication
$\cdot_{\g^*}$ is given by
$$\langle \xi\cdot_{\g^*}\eta, x\rangle=\frkB(r^\sharp(
\xi)\cdot_\g r^\sharp(\eta), x),\quad \forall \xi,\eta\in \g^*, x\in \g,$$
where $r^\sharp:\g^*\rightarrow \g$ is defined by \eqref{eq:rB}.
\end{cor}

\pf It follows form Propositions \ref{thm:pseudoHS} and
\ref{pro:dual}. \qed

\section{Para-complex quadratic pre-Lie
algebras and para-complex pseudo-Hessian pre-Lie algebras}

  In this section, we  study the role of Nijenhuis
operators on a pre-Lie algebra in another kind of geometric
structures. Explicitly, we introduce the notion of a para-complex
structure on a pre-Lie algebra and show that a para-complex
structure gives rise to a pair of transversal pre-Lie subalgebras.
In particular, we introduce the notions of a paracomplex quadratic
pre-Lie algebra and a paracomplex pseudo-Hessian pre-Lie algebra.
For the former,  we show that there is a one-to-one correspondence
between para-complex quadratic pre-Lie algebras and
para-K\"{a}hler Lie algebras.

\begin{defi}
     Let $(\g,\cdot_\g)$ be a pre-Lie algebra and $N$ a Nijenhuis operator on $(\g,\cdot_\g)$. If $N^2={\Id}$ and $\dim \ker(N+{\Id})=\dim \ker(N-{\Id})$,  then  we call $N$ a {\bf para-complex structure} on the pre-Lie algebra $\g$.
   \end{defi}

   The Nijenhuis condition is exactly the integrability condition.    Let $(\g,\cdot_\g ,N)$ be a  para-complex   pre-Lie algebra. Set
$$\g^+=\{x\in \g|N(x)=x\},\;\;\g^{-}=\{x\in\g|N(x)=-x\}.$$
Then $\dim(\g^+)=\dim(\g^-)$ and $\g=\g^+\oplus\g^-$ as the direct sum of vector spaces.
   \begin{pro}\label{pro:subalgebra}
With the above notations, both $\g^+$ and $\g^-$ are pre-Lie subalgebras of the pre-Lie algebra $(\g,\cdot_\g)$.
   \end{pro}
   \pf For all $x,y\in \g^+$, we have
$$N( x\cdot_\g y )=N(x)\cdot_\g y+x\cdot_\g N(y)-N(N(x)\cdot_\g N(y))= x\cdot_\g y+x\cdot_\g y-x\cdot_\g y=x\cdot_\g y,$$
which implies that $x\cdot_\g y\in \g^+$. Therefore $\g^+$ is a pre-Lie subalgebra. Similarly,
$\g^-$ is a pre-Lie subalgebra.\qed

\subsection{Para-complex quadratic pre-Lie algebras}

Let  $(\g,[\cdot,\cdot]_\g)$ be a Lie algebra. Recall that $\omega\in\wedge^2\g^*$ is a 2-cocycle on $\g$ if
\begin{equation}\label{eq:Lie2-form}
\omega([x,y]_\g,z)+\omega([z,x]_\g,y)+\omega([y,z]_\g,x)=0,\quad \forall~x,y,z\in\g.
\end{equation}

A {\bf symplectic  Lie algebra}, which we denote by $(\g,[\cdot,\cdot]_\g,\omega)$, is a Lie algebra  $(\g,[\cdot,\cdot]_\g)$ together with a nondegenerate
  $2$-cocycle $\omega\in\wedge^2\g^*$ on $\g$. There is a close relationship between symplectic Lie algebras and quadratic pre-Lie algebras. A {\bf quadratic pre-Lie algebra} is a pre-Lie algebra $(\g,\cdot_\g)$ equipped with a skew-symmetric nondegenerate bilinear form $\omega\in\wedge^2\g^*$, which is invariant in the sense that
  \begin{equation}
  \omega(x\cdot_\g y,z)+\omega(y,[x,z]^c)=0,\quad \forall ~ x,y,z\in\g.
  \end{equation}
Due to the following important result,  quadratic pre-Lie algebras can be viewed as the underlying algebraic structures of symplectic Lie algebras.

\begin{thm}{\rm(\cite{Chu})}\label{thm:sp}
Let $(\g,[\cdot,\cdot]_\g,\omega)$ be a symplectic Lie algebra.
Then there exists a pre-Lie algebra structure ``$\cdot_\g$'' on
$\g$ given by
\begin{equation}\label{eq:LietoLSA}
\omega(x\cdot_\g y,z)=-\omega(y,[x,z]_\g),\quad \forall ~
x,y,z\in\g,
\end{equation}
such that the sub-adjacent Lie algebra is exactly
$(\g,[\cdot,\cdot]_\g)$ itself. Furthermore,
$(\g,\cdot_\g,\omega)$  is a quadratic pre-Lie algebra.
\end{thm}

\begin{defi}
Let $(\g,[\cdot,\cdot]_\g,\omega)$ be a symplectic Lie algebra and $N$  a Nijenhuis operator on the Lie algebra $(\g,[\cdot,\cdot]_\g)$. Then $(\g,[\cdot,\cdot]_\g,\omega,N)$ is called a {\bf para-K\"{a}hler Lie algebra} if $N^2={\Id}$, $\dim \ker(N+{\Id})=\dim \ker(N-{\Id})$ and the following equality holds:
\begin{equation}\label{eq:SN}
 \omega(N(x),y)=-\omega(x,N(y)),\quad \forall~x,y\in \g.
\end{equation}
\end{defi}

Note that a Nijenhuis operator  on a Lie algebra
$(\g,[\cdot,\cdot]_\g)$ satisfies $N^2={\Id}$ and $\dim
\ker(N+{\Id})=\dim \ker(N-{\Id})$ is called a {\bf para-complex
structure} on
$(\g,[\cdot,\cdot]_\g)$.

   Now we give the notion of a para-complex quadratic pre-Lie algebra.

   \begin{defi}
     Let $(\g,\cdot_\g,\omega)$ be a quadratic pre-Lie algebra and $N$ a paracomplex structure on $(\g,\cdot_\g)$. If  \eqref{eq:SN} holds, we call $(\g,\cdot_\g,\omega,N)$ a \bf{para-complex quadratic pre-Lie algebra}.
   \end{defi}

   \begin{pro}\label{pro:quapara}
Let $(\g,\cdot_\g,\omega)$ be a quadratic pre-Lie algebra and $N$
a Nijenhuis operator on $(\g,\cdot_\g)$. Then
$(\g,\cdot_\g,\omega,N)$ is a  para-complex quadratic pre-Lie
algebra if and only if there exist two isotropic pre-Lie
subalgebras $\g^+$ and $\g^-$ (with respect to the quadratic
structure $\omega$) of the quadratic pre-Lie algebra
$(\g,\cdot_\g,\omega)$ such that $\dim(\g^+)=\dim(\g^-)$ and
$\g=\g^+\oplus\g^-$ as the direct sum of vector spaces.
\end{pro}

\pf Let $(\g,\cdot_\g,\omega,N)$ a  para-complex quadratic pre-Lie algebra.  By Proposition \ref{pro:subalgebra}, both $\g^+$ and $\g^-$ are pre-Lie subalgebras. By
 \eqref{eq:SN}, both $\g^+$ and $\g^-$ are isotropic.

Conversely, let $N:\g\rightarrow \g$ be a linear transformation
defined by
$$N(x+y)=x-y,\quad \forall x\in \g^+, y\in\g^-.$$
Then $N^2={\rm Id}$ and $N$ is a Nijenhuis
operator on the pre-Lie algebra $(\g,\cdot_\g)$.  It is obvious that
$$\dim ({\rm ker} N+{\rm Id})=\dim \g^-=\dim \g^+=\dim ({\rm ker} N-{\rm
Id}).$$ Moreover, since both $\g^+$ and $\g^-$ are isotropic, for all $x_1,x_2\in \g^+, y_1,y_2\in
\g^-$, we have
$$\omega(N(x_1+y_1), x_2+y_2)=\omega (x_1,y_2)-\omega(y_1,x_2)=-\omega (x_1+y_1, N(x_2+y_2)),$$
 which implies that \eqref{eq:SN} holds. Thus, $(\g,\cdot_\g,\omega,N)$ a  para-complex quadratic pre-Lie algebra. \qed\vspace{3mm}

  The following theorem establishes the relation between
para-complex quadratic pre-Lie algebras and para-K\"{a}hler Lie
algebras.

\begin{thm}
Let $(\g,[\cdot,\cdot]_\g,\omega,N)$ be a para-K\"{a}hler Lie algebra  and $(\g,\cdot_\g,\omega)$ the associated quadratic pre-Lie algebra given in Theorem \ref{thm:sp}. Then
\begin{itemize}
  \item[\rm(i)]$(\g, \cdot_\g,\omega,N)$ is a para-complex quadratic pre-Lie algebra;

   \item[\rm(ii)] $(\g,[\cdot,\cdot]_N,\omega)$ is a symplectic Lie algebra and $(\g,[\cdot,\cdot]_N,\omega,N)$ is also a para-K\"{a}hler Lie algebra;

    \item[\rm(iii)] $(\g,\cdot_N)$ is the corresponding pre-Lie algebra associated to the symplectic Lie algebra $(\g,[\cdot,\cdot]_N,\omega)$,  where $\cdot_N$ is given by \eqref{eq:LANij}.
  \end{itemize}
\end{thm}

\pf (i) By \eqref{eq:LietoLSA}, for all $x,y,z\in\g$, we have
\begin{eqnarray}
\nonumber\omega(x\cdot_N y,z)&=&\omega(x\cdot_\g N(y)+N(x)\cdot_\g y-N(x\cdot_\g y),z)\\
\nonumber&=&-\omega(N(y),[x,z]_\g)-\omega(y,[N(x),z]_\g)-\omega(N(x\cdot_\g y),z)\\
\nonumber&=&\omega(y,N([x,z]_\g))-\omega(y,[N(x),z]_\g)+\omega(x\cdot_\g y,N(z))\\
\nonumber&=&\omega(y,N([x,z]_\g))-\omega(y,[N(x),z]_\g)-\omega(y,[x,N(z)]_\g)\\
\label{eq:LietoLSA2}&=&-\omega(y,[x,z]_N).
\end{eqnarray}
By \eqref{eq:LietoLSA}, \eqref{eq:LietoLSA2}, $N^2={\Id}$ and the fact that $N$ is a Nijenhuis operator on the Lie algebra $(\g,[\cdot,\cdot]_\g)$, we have
\begin{eqnarray*}
&&\omega(y,N(x)\cdot_\g N(z)-N(x\cdot_N z))\\&=&\omega(y,N(x)\cdot_\g N(z))+\omega(N(y),x\cdot_N z)\\
&=&\omega(N([N(x),y]_\g),z)-\omega([x,N(y)]_N,z)\\
&=&\omega(N([N(x),y]_\g)-[N(x),N(y)]_\g-[x,N^2(y)]_\g+N([x,N(y)]_\g),z)\\
&=&\omega\Big(N([N(x),y]_\g)-N\big([N(x),y]_\g+[x,N(y)]_\g-N([x,y]_\g)\big)-[x,N^2(y)]_\g+N([x,N(y)]_\g),z\Big) \\
 &=&\omega(N^2([x, y]_\g)-[x, N^2(y)]_\g,z)\\
&=&0,
\end{eqnarray*}
which implies that  $N$ is a Nijenhuis operator on the pre-Lie algebra $(\g,\cdot_\g)$. Furthermore, it is obvious that $(\g, \cdot_\g,\omega,N)$ is a para-complex quadratic pre-Lie algebra.

(ii) By \eqref{eq:LietoLSA2} and $[x,y]_N=x\cdot_N y-y\cdot_N x$, we have
\begin{eqnarray*}
\dM^N\omega(x,y,z)&=&\omega([x,y]_N,z)+\omega([z,x]_N,y)+\omega([y,z]_N,x)\\
&=&\omega(x\cdot_N y-y\cdot_N x,z)+\omega([z,x]_N,y)+\omega([y,z]_N,x)\\
&=&\omega(x\cdot_N y,z)+\omega(y,[x,z]_N)-\omega(y\cdot_N x,z)-\omega(x,[y,z]_N)=0,
\end{eqnarray*}
which implies that $\omega$ is $2$-cocycle on the Lie algebra $(\g,[\cdot,\cdot]_N)$. Thus $(\g,[\cdot,\cdot]_N,\omega)$ is a symplectic Lie algebra. Furthermore, since $N$ is also a Nijenhuis operator on the Lie algebra $(\g,[\cdot,\cdot]_N)$, $(\g,[\cdot,\cdot]_N,\omega,N)$ is also a para-K\"{a}hler Lie algebra;

(iii)  By \eqref{eq:LietoLSA2}, $(\g,\cdot_N)$ is the corresponding pre-Lie algebra associated to the symplectic Lie algebra $(\g,[\cdot,\cdot]_N,\omega)$.
\qed\vspace{3mm}

In  \cite{Ova}, the author gave a
classification of 4-dimensional real symplectic Lie algebras. Now
we consider the following example of 4-dimensional para-K\"{a}hler
Lie algebras and the corresponding para-complex quadratic pre-Lie
algebras. See \cite{Cal,Cal1} for more details on the
classification of 4-dimensional para-K\"{a}hler Lie algebras.

\begin{ex}{\rm
Let $\{e_1,e_2,e_3,e_4\}$ be a basis of a $4$-dimensional
symplectic Lie algebra  $\g$. The Lie algebra structure is given
by   the following non-zero brackets
$$[e_1,e_3]_\g=e_3,\quad [e_1,e_4]_\g=e_4,\quad[e_2,e_3]_\g=-e_4,\quad[e_2,e_4]_\g=e_3,$$
and the symplectic structure $\omega$ is given by $\omega=e^1\wedge e^3+e^2\wedge e^4$, where $e^i,i=1,\ldots,4$ is the dual basis of $e_i$.
The corresponding pre-Lie algebra structure is given by
\begin{eqnarray*}
&&e_1\cdot_\g e_1=-e_1,\quad e_2\cdot_\g e_2=e_1,\quad e_1\cdot_\g e_2=e_2\cdot_\g e_1=-e_2,\\
&&e_3\cdot_\g e_1=-e_3,\quad e_3\cdot_\g e_2=e_4,\quad e_4\cdot_\g e_1=-e_4,\quad e_4\cdot_\g e_2=-e_3.
\end{eqnarray*}
Let
$
N_1 =\begin{bmatrix}1& 0&0&0\\ 0&1&0&0\\ 0&0&-1&0\\0&0&0&-1\end{bmatrix},\quad N_2 =\begin{bmatrix}-1& 0&0&0\\ 0&-1&0&0\\ 0&0&1&0\\0&0&0&1\end{bmatrix}.
$
 Then $(\g,[\cdot,\cdot]_\g,\omega,N_i)$, $i=1,2$, are para-K\"{a}hler Lie algebras and $(\g,\cdot_\g,\omega,N_i)$ are the corresponding para-complex quadratic pre-Lie algebras.}
\end{ex}

\subsection{Para-complex pseudo-Hessian pre-Lie algebras}

\begin{defi}
  Let $N$ a para-complex structure and $\frkB$ be a pseudo-Hessian structure   on a pre-Lie algebra $(\g,\cdot_\g)$. Then $(\frkB,N)$ is called a {\bf para-complex pseudo-Hessian structure} on the pre-Lie algebra  $(\g,\cdot_\g)$ if
  \begin{equation}
 \frkB(N(x),y)=-\frkB(x,N(y)),\quad \forall~x,y\in \g.
\end{equation}
\end{defi}

Let $(\frkB,N)$ be a   paracomplex pseudo-Hessian structure  on the pre-Lie algebra  $(\g,\cdot_\g)$.
 Then we can obtain a skew-symmetric bilinear form $\omega\in\wedge^2\g^*$ by
 $$
 \omega(x,y):=\frkB(x,N(y)),\quad \forall~x,y\in \g.
 $$
Furthermore, $\omega$ and the para-complex structure $N$ are compatible in the sense that
$$
\omega(N(x),N(y))=-\omega(x,y).
$$

Similar as Proposition \ref{pro:quapara}, we have

\begin{pro}
Let $(\g,\cdot_\g,\frkB)$ be a pseudo-Hessian pre-Lie algebra and
$N$ a paracomplex structure on a pre-Lie algebra $(\g,\cdot_\g)$.
Then  $(\g,\cdot_\g,\frkB,N)$ is a  para-complex pseudo-Hessian
pre-Lie algebra if and only if there exist two isotropic pre-Lie
subalgebras $\g^+$ and $\g^-$ (with respect to the symmetric
bilinear form $\frkB$) of the pseudo-Hessian pre-Lie algebra $
(\g,\cdot_\g,\frkB)$    such that $\dim(\g^+)=\dim(\g^-)$ and
$\g=\g^+\oplus\g^-$ as the direct sum of vector spaces.

\end{pro}

\pf The proof is totally parallel to that of Proposition
\ref{pro:quapara}, we leave it to  readers. \qed

\begin{rmk}
In \cite{NiBai}, we call the above structure a {\bf Manin triple
of pre-Lie algebras associated to a nondegenerate symmetric
2-cocycle} which is equivalent to a bialgebra structure, namely,
L-dendriform bialgebra.
\end{rmk}
\emptycomment{
The following example says not every Nijenhuis operator $N$ on a symplectic Lie algebra is a Nijenhuis operator on the corresponding pre-Lie algebra.
\begin{ex}
Let $\{e_1,e_2,e_3,e_4\}$ be a basis of a $4$-dimensional symplectic Lie algebra  $\g$. The Lie algebra structure is  given by
$$[e_1,e_2]_\g=e_4$$
and the symplectic structure $\omega$ is given by $\omega=e^1\wedge e^3+e^2\wedge e^4$.

The corresponding pre-Lie algebra structure is given by
$$e_1\cdot_\g e_2=e_4,\quad e_2\cdot_\g e_2=-e_3.$$
By direct calculation, we can verify that
$$
N(\g)=\Big\{\begin{bmatrix}a& b&c&0\\ 0&0&0&0\\ d&e&-a&0\\e&f&-b&0\end{bmatrix}\Big|\forall~ a,b,c,d,e,f\in\Real\Big\}
$$
are Nijenhuis operators for the above Lie algebra. It is not hard to see that the only following subsets in $N(\g)$ are  Nijenhuis operators for the corresponding pre-Lie algebra
$$
N_1(\g)=\Big\{\begin{bmatrix}0& a&0&0\\ 0&0&0&0\\ b&c&0&0\\c&d&-a&0\end{bmatrix}\Big|\forall~ a,b,c,d\in\Real\Big\}
$$
and
$$
N_2(\g)=\Big\{\begin{bmatrix}a&b&c&0\\ 0&0&0&0\\-\frac{a^2}{c}&-\frac{ab}{c}&-a&0\\-\frac{ab}{c}&d&-b&0\end{bmatrix}\Big|\forall~ a,b,c,d\in\Real,c\neq 0\Big\}.
$$
\end{ex}
}

\section{Some examples of Nijenhuis operators on pre-Lie algebras}

\subsection{Pre-Lie algebras
from the operator form of the classical Yang-Baxter equation in
Lie algebras}

\begin{lem}{\rm(\cite{GolSok})}
Let $(\g,[\cdot,\cdot]_\g)$ be a Lie algebra and $r:\g\rightarrow \g$ a linear map satisfying the
operator form of the classical Yang-Baxter equation
\begin{equation}
[r(x),r(y)]_\g=r([r(x),y]_\g+[x,r(y)]_\g),\quad \forall~x,y\in\g.
\end{equation}
Then $(\g,\cdot^r)$ is a pre-Lie algebra, where $\cdot^r$ is defined by
\begin{equation}\label{eq:RotaLA}
x\cdot^r y=[r(x),y]_\g,\quad \forall~x,y\in\g.
\end{equation}
\end{lem}

\begin{pro}
  Let $(\g,[\cdot,\cdot]_\g)$ be a Lie algebra. Let $r:\g\rightarrow \g$ be a linear map satisfying the
operator form of the classical Yang-Baxter equation in
$(\g,[\cdot,\cdot]_\g)$ and $N$ a Nijenhuis operator on
$(\g,[\cdot,\cdot]_\g)$. If $N\circ r=r\circ N$,
 then
$r$ also satisfies the operator form of the classical Yang-Baxter
equation in the Lie algebra $(\g,[\cdot,\cdot]_N)$. Furthermore,
$N$ is a Nijenhuis operator on the pre-Lie algebra $(\g,\cdot^r)$
given by \eqref{eq:RotaLA} and satisfies
\begin{equation}\label{eq:NR}
x\cdot_{N}^r y=[r(x),y]_N,\quad \forall~x,y\in\g.
\end{equation}
\end{pro}
\pf By direct calculation, we have
\begin{eqnarray*}
[r(x),r(y)]_N&=&[N(r(x)),r(y)]_\g+[r(x),N(r(x))]_\g-N([r(x),r(y)]_\g)\\
&=&[r(N(x)),r(y)]_\g+[r(x),r(N(y))]_\g-N([r(x),r(y)]_\g)\\
&=&r([r(N(x)),y]_\g+[N(x),r(y)]_\g)+r([r(x),N(y)]_\g\\
&&+[x,r(N(y))]_\g)-r(N([r(x),y]_\g)+N([x,r(y)]_\g))\\
&=&r\Big([N(r(x)),y]_\g+[r(x),N(y)]_\g-N([r(x),y]_\g)\\
&&+[N(x),r(y)]_\g+[x,N(r(y))]_\g-N([x,r(y)]_\g)\Big)\\
&=&r([r(x),y]_N+[x,r(y)]_N),
\end{eqnarray*}
which implies that  $r$ also satisfies the operator form of the
classical Yang-Baxter equation in  the Lie algebra
$(\g,[\cdot,\cdot]_N)$.

Furthermore, we have
\begin{eqnarray*}
N(x)\cdot^r N(y)&=&[r(N(x)),N(y)]_\g=[N(r(x)),N(y)]_\g\\
&=&N([r(N(x)),y]_\g+[r(x),N(y)]_\g-N([r(x),y]_\g))\\
&=&N(N(x)\cdot^r y+x\cdot^r N(y)-N(x\cdot^r y)),
\end{eqnarray*}
which implies that $N$ is a Nijenhuis operator on the pre-Lie algebra $(\g,\cdot^r)$ and \eqref{eq:NR} holds. \qed

\subsection{Pre-Lie algebras associated to integrable (generalized) Burgers equation}
Let $\g$ be a vector space  with an ordinary scalar product $\langle\cdot,\cdot\rangle$ and $a\in \g$. Then
\begin{equation}
x\cdot^a y=\langle x,y\rangle a+\langle x,a\rangle y,\quad\forall~x,y\in \g,
\end{equation}
defines a pre-Lie algebra structure on $\g$.
\begin{rmk}
This pre-Lie algebra plays an important role in the integrable (generalized) Burgers equation \cite{BurgerEq2}.
\end{rmk}

\begin{pro}
Let $(\g,\cdot^a)$ be the pre-Lie algebra given above.
Then any $N\in\gl(\g)$ is a Nijenhuis operator on the sub-adjacent Lie algebra $(\g,[\cdot,\cdot]^c)$ of the pre-Lie algebra $(\g,\cdot^a)$. Furthermore,
if $N\in\gl(\g)$  satisfies
\begin{equation}\label{eq:Burger2}
\langle N(x),y\rangle=-\langle x,N(y)\rangle,\quad \forall~x,y\in \g,
\end{equation}
then $N$ is a Nijenhuis operator on the pre-Lie algebra  $(\g,\cdot^a)$ if and only if $$\langle N(x),N(y)\rangle a=-\langle x,y\rangle N^2(a).$$
\end{pro}
\pf For all $x,y\in \g$ and $N\in\gl(\g)$, we have
\begin{eqnarray*}
&&[N(x),N(y)]^c-N([N(x),y]^c+[x,N(y)]^c-N([x,y]^c))\\
&=&\langle N(x),a\rangle N(y)-\langle N(y),a\rangle N(x)-N\Big(\langle N(x),a\rangle y-\langle y,a\rangle N(x)\\
&&+\langle x,a\rangle N(y)-\langle N(y),a\rangle x-\langle x,a\rangle N(y)+\langle y,a\rangle N(x)\Big)\\
&=&0.
\end{eqnarray*}
Thus, any $N\in\gl(\g)$ is a Nijenhuis operator on the sub-adjacent Lie algebra $(\g,[\cdot,\cdot]^c)$.

For any $N$ satisfying \eqref{eq:Burger2}, we have
\begin{eqnarray*}
&&N(x)\cdot^a N(y)-N(N(x)\cdot^a y+x\cdot^aN(y)-N(x\cdot^a y))\\
&=&\langle N(x),N(y)\rangle a+\langle N(x),a\rangle N(y)-N\Big(\langle N(x),y\rangle a+\langle N(x),a\rangle y\\
&&+\langle x,N(y)\rangle a+\langle x,a\rangle N(y)-\langle x,y\rangle N(a)-\langle x,a\rangle N(y)\Big)\\
&=&\langle N(x),N(y)\rangle a+\langle x,y\rangle N^2(a).
\end{eqnarray*}
The second conclusion follows immediately.\qed

\begin{pro}
Let $(\g,\cdot^a)$ be the pre-Lie algebra given above and $N\in\gl(\g)$  satisfying \eqref{eq:Burger2}.
 Then we have
 $$
 x\cdot^a_N(y)=-x\cdot^{N(a)}y, \quad \forall~x,y\in \g.
 $$
Thus, $(\g,\cdot^a_N)$ is a pre-Lie algebra.
\end{pro}
\pf By \eqref{eq:Burger2}, we have
\begin{eqnarray*}
x\cdot_{N}^a y&=& N(x) \cdot^a y+x \cdot^a N(y)-N(x \cdot^a y)\\
&=&\langle N(x),y\rangle a+\langle N(x),a\rangle y+\langle x,N(y)\rangle a+\langle x,a\rangle N(y)-\langle x,y\rangle N(a)-\langle x,a\rangle N(y)\\
&=&\langle x,y\rangle (-N(a))+\langle x,-N(a)\rangle y\\
&=&-x\cdot^{N(a)}y.
\end{eqnarray*}
Thus, $(\g,\cdot^a_N)$ is a pre-Lie algebra. \qed
\subsection{Pre-Lie algebras
from Rota-Baxter operators on associative algebras}

  Let $(\g,\ast)$ be an associative algebra. Then a linear
map $N:\g\longrightarrow \g$ is called a Nijenhuis operator on
$\g$ if it is a Nijenhuis operator as a pre-Lie algebra. For more
details of Nijenhuis operators on associative algebras, see
$\cite{CaGraMar,Fard}$.

\begin{lem}{\rm(\cite{GolSok})}
Let $(\g,\ast)$ be an associative algebra and $R:\g\longrightarrow \g$  a linear map satisfying
\begin{equation}
R(x)\ast R(y)+R(x\ast y)=R(R(x)\ast y+x\ast R(y)),\quad\forall~x,y\in\g.
\end{equation}
Then $(\g,\cdot^R)$ is a pre-Lie algebra, where $\cdot^R$ is given by
\begin{equation}\label{eq:AssNLA}
x\cdot^R y=R(x)\ast y-y\ast R(x)-x\ast y,\quad \forall~x,y\in \g.
\end{equation}
\end{lem}

  \begin{rmk}The above $R$ is exactly  a Rota-Baxter
operator of weight $-1$ by regarding $(\g,\ast)$ as a pre-Lie
algebra. In fact, the notion of Rota-Baxter operator on an
associative algebra was introduced to solve analytic \cite{Bax}
and combinatorial \cite{Rota} problems.
\end{rmk}

\begin{pro}
Let  $R$ be a Rota-Baxter operator of weight $-1$ and $N$  a
Nijenhuis operator on an associative algebra $(\g,\ast)$. If
$R\circ N=N\circ R$, then we have
 \begin{itemize}
  \item[\rm(i)]$R$ is also a  Rota-Baxter operator of weight $-1$ on the associative algebra $(\g,\ast_N)$;

  \item[\rm(ii)] $N$ is a Nijenhuis operator on the pre-Lie algebra $(\g,\cdot^R)$;

    \item[\rm(iii)]the pre-Lie algebra $(\g,\cdot^R_N)$ defined by the Nijenhuis operator $N$ is exactly the pre-Lie algebra given by  \eqref{eq:AssNLA} associated to the associative algebra $(\g,\ast_N)$ and $R$.
 \end{itemize}
\end{pro}
\pf Let $R$ be a Rota-Baxter operator and $N$  a Nijenhuis operator on $\g$ satisfying $R\circ N=N\circ R$. Then we have
\begin{eqnarray*}
&&R(x)\ast_N R(y)+R(x\ast_N y)-R(R(x)\ast_N y+x\ast_N R(y))\\
&=&N(R(x))\ast R(y)+R(x)\ast N(R(y))-N(R(x)\ast R(y)))+R(N(x) \ast y+x\ast N(y)-N(x\ast y))\\
&&-R\Big(N(R(x))\ast y+R(x)\ast N(y)-N(R(x)\ast y)+N(x)\ast R(y)+x\ast N(R(y))-N(x\ast R(y))\Big)\\
&=&R(N(x))\ast R(y)+R(N(x)\ast y)-R(R(N(x))\ast y+N(x)\ast R(y))\\
&&+R(x)\ast R(N(y))+R(x\ast N(y))-R(R(x)\ast N(y)+x\ast R(N(y)))\\
&&-N\Big(R(x)\ast R(y)+R(x\ast y)-R(R(x)\ast y-x\ast R(y))\Big)=0,
\end{eqnarray*}
which implies that $R$ is a Rota-Baxter operator  of
weight $-1$  on the associative algebra $(\g,\ast_N)$.

For all $x,y\in \g$, also by $R\circ N=N\circ R$, we have
\begin{eqnarray*}
&&N(x)\cdot^R N(y)-N(N(x)\cdot^R y+x\cdot^R N(y)-N(x\cdot^R y))\\
&=&R(N(x))\ast N(y)-N(y)\ast R(N(x))-N(x)\ast N(y)-N\big(R(N(x))\ast y-y\ast R(N(x))\\&&-N(x)\ast y
+R(x)\ast N(y)-N(y)\ast R(x)-x\ast N(y)-N(R(x)\ast y-y\ast R(x)-x\ast y)\big)\\
&=&N(R(x))\ast N(y)-N(N(R(x))\ast y+R(x)\ast N(y)-N(R(x)\ast y))\\
&&-N(y)\ast N(R(x))+N(N(y)\ast R(x)+y\ast N(R(x))-N(y\ast R(x)))\\
&&-N(x)\ast N(y)+N(N(x)\ast y+x\ast N(y)-N(x\ast y))=0,
\end{eqnarray*}
which implies that $N$ is a Nijenhuis operator on the pre-Lie algebra $(\g,\cdot^R)$.

At last, we have
\begin{eqnarray*}
x\cdot^R_N y&=&N(x)\cdot^R y+x\cdot^R N(y)-N(x\cdot^R y)\\
&=&R(N(x))\ast y-y\ast R(N(x))-N(x)\ast y+R(x)\ast N(y)-N(y)\ast R(x)-x\ast N(y)\\
&&-N(R(x)\ast y-y\ast R(x)-x\ast y)\\
&=&R(x)\ast_N y-y\ast_N R(x)-x\ast_N y,
\end{eqnarray*}
which finishes the proof.\qed

\subsection{Novikov algebras
from derivations on commutative associative algebras}

Recall that a Novikov algebra $(\g,\cdot_\g)$ is a pre-Lie algebra
satisfying $R_xR_y = R_yR_x$ for all $x, y \in \g.$ The following
construction of Novikov algebras is due to   Gel'fand
\rm\cite{GelDor},   Filipov \rm\cite{Fili} and    Xu \rm\cite{Xu}.
\begin{lem}
Let $D$ be a derivation on a commutative associative algebra $(\g,\ast)$  over a field $\K$.
Then for all $s\in\K$, the new product
\begin{equation}\label{eq:AssDer1}
x\cdot^s y=x\ast D(y)+s (x\ast y),\quad \forall~x,y\in \g
\end{equation}
makes $(\g,\cdot^s)$ being a Novikov algebra. Furthermore, for all $\alpha \in \g$, $(\g,\cdot^\alpha)$ is also a Novikov algebra, where $\cdot^\alpha$ is given by
$$
x\cdot^\alpha y=x\ast D(y)+\alpha\ast x\ast y,\quad \forall~x,y\in \g.
$$
\end{lem}

\begin{pro}
 Let $D$ be a derivation and $N$  a Nijenhuis operator on a commutative associative algebra $(\g,\ast)$ over the field  $\K$. If $D\circ N=N\circ D$, then we have
  \begin{itemize}
 \item[\rm(i)] $D$ is also a  derivation on the commutative associative algebra $(\g,\ast_N)$;
 \item[\rm(ii)] $N$ is a Nijenhuis operator on the pre-Lie algebra $(\g,\cdot^s)$ for $s\in \K$;
  \item[\rm(iii)]the pre-Lie algebra $(\g,\cdot_{N}^s)$ defined by the Nijenhuis operator $N$
is exactly the pre-Lie algebra given by  \eqref{eq:AssDer1} associated to the commutative associative algebra $(\g,\ast_{N})$, i.e.
$$
x\cdot_{N}^s y=x\ast_{N} D(y)+s (x\ast_{N} y). \quad\forall~ x,y\in \g.
$$
\end{itemize}
\end{pro}
\pf By $D\circ N=N\circ D$, we have
\begin{eqnarray*}
D(x\ast _N y)&=&D(N(x)\ast y+x\ast N(y)-N(x\ast y))\\
&=&D(N(x))\ast y+N(x) \ast D(y)+D(x)\ast N(y)\\
&&+x\ast D(N(y))-N(D(x))\ast y+x\ast D(y))\\
&=&N(D(x))\ast y+D(x)\ast N(y)-N(D(x)\ast y)\\
&&+N(x) \ast D(y)+x\ast N(D(y))-N(x\ast D(y))\\
&=&D(x)\ast _N y+x\ast_N D(y),
\end{eqnarray*}
which implies that $D$ is a derivation on the commutative associative algebra $(\g,\ast_N)$.

Also by $D\circ N=N\circ D$, we have
\begin{eqnarray*}
&&N(x)\ast^s N(y)-N\big(N(x)\ast^s y+x\ast^s N(y)-N(x\ast^s y)\big)\\
&=&N(x)\ast D(N(y))+s (N(x)\ast N(y))-N\Big(N(x)\ast D(y)\\
&&+s N(x)\ast y+x\ast D(N(y))+s x\ast N(y)-N(x\ast D(y)+s (x\ast y))\Big)\\
&=&N(x)\ast N(D(y))-N\big(N(x)\ast D(y)+x\ast N(D(y))-N(x\ast D(y))\big)\\
&&+s \Big(N(x)\ast N(y)-N\big( N(x)\ast y+x\ast N(y)-N( x\ast y)\big)\Big)=0,
\end{eqnarray*}
which implies that $N$ is a Nijenhuis operator on the pre-Lie algebra $(\g,\cdot^s)$ for all $s\in \Field$.

At last, we have
\begin{eqnarray*}
x\cdot^s_N y&=&N(x)\cdot^s y+x\cdot^s N(y)-N(x\cdot^s y)\\
&=&N(x)\ast D(y)+s (N(x)\ast y)+x\ast D(N(y))+s (x\ast  N(y))-N(x\ast D(y)+s (x\ast y))\\
&=&x\ast_{N} D(y)+s (x\ast_{N} y),
\end{eqnarray*}
which finishes the proof.\qed\vspace{3mm}

Similarly, we have
\begin{pro}
  For a fixed $\alpha\in \g$, assume that $N(\alpha\ast x\ast y)=\alpha\ast N(x\ast y)$ for all $x,y\in \g$, then the above results are also true.
\end{pro}

Qi Wang and Yunhe Sheng

Department of Mathematics, Jilin University, Changchun 130012, Jilin, China

Email: shengyh@jlu.edu.cn
\vspace{2mm}

Chengming Bai

Chern Institute of Mathematics and LPMC, Nankai University,
Tianjin 300071, China

Email: baicm@nankai.edu.cn\vspace{2mm}

Jiefeng Liu

Department of Mathematics, Xinyang Normal University, Xinyang 464000, Henan, China

 Email:liujf12@126.com
\end{document}